\documentclass[10pt]{amsart}

\usepackage{fullpage}
\usepackage{amssymb}
\usepackage{amsmath}
\usepackage{amsxtra}
\usepackage{amscd}
\usepackage{graphicx}
\usepackage{amsfonts}
\usepackage{pb-diagram}
\usepackage{float}
\usepackage{enumerate}

\numberwithin{equation}{section}

\newtheorem{thm}{Theorem}
\newtheorem{lem}{Lemma}
\newtheorem{cor}{Corollary}
\newtheorem{prop}{Proposition}

\newtheorem{defn}{Definition}

\newtheorem{rem}{Remark}

\newtheorem*{notation}{Notation}

\def\a{\alpha}

\def\s{\sigma}
\def\leq{\leqslant}

\def\Z{{\mathbb Z}}

\begin{document}

\title[The Yokonuma-Hecke algebras and the HOMFLYPT polynomial]
  {The Yokonuma-Hecke algebras and the HOMFLYPT polynomial}
\author{Maria Chlouveraki}
\address{Laboratoire de Math\'ematiques UVSQ, B\^atiment Fermat, 45 avenue des \'Etats-Unis,  78035 Versailles cedex, France.}
\email{maria.chlouveraki@uvsq.fr}

\author{Sofia Lambropoulou}
\address{Department of Mathematics,
National Technical University of Athens,
Zografou campus,
{GR-15780} Athens, Greece.}
\email{sofia@math.ntua.gr}


\thanks{The research project is implemented within the framework of the Action ``Supporting Postdoctoral Researchers'' of the Operational Program
``Education and Lifelong Learning'' (Action's Beneficiary: General Secretariat for Research and Technology), and is co-financed by the European Social Fund (ESF) and the Greek State. The first author would also like to thank Guillaume Pouchin for fruitful conversations and comments.}


\begin{abstract}
We compare the invariants for classical knots and links defined
using the Juyumaya trace on the Yokonuma-Hecke algebras with  the HOMFLYPT polynomial. We show that these invariants do not coincide with the HOMFLYPT except in a few trivial cases.
\end{abstract}

\maketitle

\section*{Introduction}

The Yokonuma-Hecke  algebras ${\rm Y}_{d,n}(u)$ were introduced by Yokonuma \cite{yo} in the context of Chevalley groups, as generalizations of the Iwahori-Hecke algebras. The algebras ${\rm Y}_{d,n}(u)$ may be also viewed as quotients of the framed braid group algebra over a quadratic relation (see (\ref{quadr})) involving the framing generators by means of certain weighted idempotents $e_{i}$. Thus the classical braid groups are also represented in the algebras ${\rm Y}_{d,n}(u)$.

In \cite{ju} Juyumaya constructed a unique linear Markov trace ${\rm tr}$ on the algebras ${\rm Y}_{d,n}(u)$, depending on $d$ parameters, $z, x_1, \ldots ,x_{d-1}$. The trace ${\rm tr}$ was used subsequently in \cite{jula2} for defining isotopy invariants for framed knots. As it turned out, the trace ${\rm tr}$ would not re-scale directly according to the braid equivalence moves. Therefore, certain conditions had to be imposed, implying that the trace parameters $x_1, \ldots ,x_{d-1}$ had to satisfy a non linear system of equations, the so-called E-\emph{system} (see  (\ref{Esystem})). G\'{e}rardin proved that the solutions of the E-system are parametrized by the non-empty subsets $S$ of ${\mathbb Z}/d{\mathbb Z}$ (see Appendix of  \cite{jula2}). Given now any solution of the E-system, 2-variable isotopy invariants for framed, classical and singular knots  were constructed respectively in \cite{jula2,jula3,jula4}.

For classical knots we have the well-known HOMFLYPT or 2-variable Jones polynomial $P$  \cite{jo}, which is determined by the Ocneanu trace (with parameter $\zeta$) on the Iwahori-Hecke algebras $\mathcal{H}_n(q)$ of type $A$. Therefore, it is natural to ask how the invariant $P$ compares with every invariant $\Delta_S$ derived from the Juyumaya trace on the algebras ${\rm Y}_{d,n}(u)$, for any $d \in \mathbb{N}$ and for any non-empty subset $S$ of ${\mathbb Z}/d{\mathbb Z}$. Computational data so far do not indicate that one invariant is topologically stronger than the other (see \cite{CJJKL}).

In order to compare the knot invariants $P$ and $\Delta_S$, we would like to be able to specialize the indeterminates $x_1, \ldots, x_{d-1}$ to a solution of the {\rm E}-system as early as possible during the construction of $\Delta_S$. This goal is achieved in Section 3 with the construction of the linear map $\varphi$ on ${\rm Y}_{d,n}(u)$.
However,  as we show, there is no appropriate algebra homomorphism between ${\rm Y}_{d,n}(u)$ and $\mathcal{H}_n(q)$, unless ${E:={\rm tr}(e_i)=1}$, and this
 makes a connection between the corresponding trace functions impossible. In this paper we prove that the invariants $P$ and $\Delta_S$  do not coincide except in a few trivial cases, that is, $u=1$ or $q=1$ or $E=1$, given by Theorem~\ref{mainthm}. In fact we show (Theorem~\ref{notscalar}) that these are the only cases where the one invariant is a scalar multiple of the other, with scalars in $\mathbb{C}(q,\zeta,u,z,E)$.

The paper is organized as follows: In the first two sections we present some preliminary results on Iwahori-Hecke algebras and Yokonuma-Hecke algebras. In Section~3 we introduce the {\it specialized Juyumaya trace}, where the indeterminates $x_1, \ldots, x_{d-1}$ specialize to complex numbers,  and we show that it factors through the linear map $\varphi$ that we construct. We compare the specialized Juyumaya trace with the Ocneanu trace  and we obtain one case  (when $E=1$) where the invariants $P$ and $\Delta_S$ coincide. In the last two sections we proceed with comparing further the invariants, in order to obtain all cases where they coincide. More precisely, in Section~4 we give some necessary conditions for  $P$ and $\Delta_S$ to coincide, by evaluating the invariants on specific braid words.
Finally, in Section~5 we prove, with the use of an elaborate induction, that these conditions are also sufficient.

\section{The 2-variable Jones or HOMFLYPT polynomial}

\subsection{\it The symmetric group} The symmetric group $\mathfrak{S}_n$ is generated by the transpositions $s_1,s_2, \ldots,$ $s_{n-1}$, with $s_i=(i, i+1)$, subject to the relations:
\begin{equation}\label{symgroup}
\begin{array}{rclcl}
  s_is_j & = & s_js_i && \mbox{for $\vert i-j\vert > 1$}\,;\\
  s_{i+1}s_is_{i+1} & = & s_is_{i+1}s_i && \mbox{for all $i$}\,;\\
 s_i^2  & =  & 1 && \mbox{for all $i$}\,.
\end{array}
\end{equation}

Let $S=\{s_1,s_2,\ldots,s_{n-1}\}$ and let $w \in \mathfrak{S}_n$. Then
$w=s_{i_1}s_{i_2}\ldots s_{i_r}$, with $s_{i_j} \in S$, is an \emph{expression} for $w$. If $r$ is minimal such that there exists an expression $w=s_{i_1}s_{i_2}\ldots s_{i_r}$, then this expression is called \emph{reduced} and $r$ is called the {\em length of } $w$. We denote the length of $w$ by $\ell(w)$.

\subsection{\it Conjugacy classes of  $\mathfrak{S}_n$}\label{conj}
The conjugacy classes of $\mathfrak{S}_n$ correspond to the cycle types of permutations; that is, two elements of $\mathfrak{S}_n$ are conjugate in $\mathfrak{S}_n$ if and only if they consist of the same number of disjoint cycles of the same lengths.
It is well-known that the conjugacy classes of $\mathfrak{S}_n$ are naturally parametrized by the partitions $\mu$ of $n$. If $\mu$  has non-zero parts $\mu_1$, $\mu_2$, \ldots, then we take $w_\mu:= s_{i_k}\ldots s_{i_2}s_{i_1}$ as representative in the class labelled by $\mu$, where $\{i_1,i_2,\ldots,i_k\}$ is the set obtained from $\{1,2,\ldots,n\}$ by removing the integers $\mu_1$, $\mu_1+\mu_2$,  $\mu_1+\mu_2+\mu_3$, $\ldots$ .  For example, if $n = 8$ and $\mu = (4,3,1)$, then $\mu_1=4$, $\mu_1+\mu_2=7$ and $\mu_1+\mu_2+\mu_3=8$, whence $w_\mu = s_6s_5s_3s_2s_1$. The point about choosing these representatives is that $w_\mu$ has minimal length in its conjugacy class, that is, we have $\ell(w_\mu) \leq \ell(w)$ for any $w \in \mathfrak{S}_n$ which is conjugate to $w_\mu$. Now let
$$\mathfrak{D}:=\{s_{i_k}\ldots s_{i_2}s_{i_1} \,|\, i_1<i_2<\cdots<i_k\} \subset \mathfrak{S}_n.$$
We obviously have $w_\mu \in \mathfrak{D}$ for every partition $\mu$ of $n$. Moreover, if $w=s_{i_k}\ldots s_{i_2}s_{i_1}  \in \mathfrak{D}$, one can easily check that, because of its cycle type, $w$ has minimal length in its conjugacy class.\footnote{This result also follows from \cite[Lemma 3.1.14]{gp}, since $w$ is a Coxeter element of the parabolic subgroup $W_J$ of $\mathfrak{S}_n$, where $J=\{s_{i_1},s_{i_2},\ldots, s_{i_k}\}\subseteq S$.}

The following result, which relates elements of minimal length in a conjugacy class, will be useful in Subsection \ref{similarities1}.

\begin{thm}\cite[Theorem 3.2.9]{gp}\label{Cmin}\
Let $C$ be a conjugacy class of $\mathfrak{S}_n$ and let $w,w'$ be two elements of minimal length in $C$. Then $w$ and $w'$ are ``strongly conjugate'' in $\mathfrak{S}_n$, that is, there exists a finite sequence $w_0=w, w_1, \ldots, w_r=w'$ such that, for all $i=0,1,\ldots,r-1$,
$$\ell(w_{i})=\ell(w_{i+1}),\,\,\, w_{i+1}=x_iw_{i}x_i^{-1} \,\, \text{ and }\,\,  \ell(x_iw_i)=\ell(x_i)+\ell(w_i)\,\, \text{ or }\,\, \ell(w_ix_i^{-1})=\ell(w_i)+\ell(x_i^{-1})$$
for some elements $x_i \in \mathfrak{S}_n$.
\end{thm}

\subsection{\it The Iwahori-Hecke algebra}

Let $q \in \mathbb{C}\setminus \{0\}$.
The Iwahori-Hecke algebra $\mathcal{H}_n(q)$ of type $A$ is the $\mathbb{C}$-associative algebra with presentation on generators $G_1$, $G_2$, $\ldots$, $G_{n-1}$, and  relations:
\begin{equation}\label{hecke}
\begin{array}{ccrclcl}
\mathrm{(b}_1)& & G_iG_j & = & G_jG_i && \mbox{for $\vert i-j\vert > 1$}\,;\\
\mathrm{(b}_2)& & G_{i+1}G_iG_{i+1} & = & G_iG_{i+1}G_i && \mbox{for all $i$}\,;\\
\mathrm{(h})& & G_i^2  & =  &  (q-1)G_i+q  && \mbox{for all $i$}\,.
\end{array}
\end{equation}

The relations $({\rm b}_1)$ and $({\rm b}_2)$  are defining relations for the classical Artin braid group $B_n$; hence $\mathcal{H}_n(q)$ can be viewed as the quotient of the group algebra ${\mathbb C}B_n$ over the quadratic relations~$({\rm h})$.
Moreover, for $q=1$ the algebra  $\mathcal{H}_n(1)$ is isomorphic to the group algebra of the symmetric group $\mathbb{C}\mathfrak{S}_n$.
If $w \in \mathfrak{S}_n$ and $w=s_{i_1}s_{i_2}\ldots s_{i_r}$ is a reduced expression, we set
$G_w:=G_{i_1}G_{i_2}\ldots G_{i_r}$. The set
$$\{G_w\,|\, w \in \mathfrak{S}_n\}$$
is the ``standard'' $\mathbb{C}$-basis of  $\mathcal{H}_n(q)$.
Now, the following set forms another linear $\mathbb{C}$-basis for  $\mathcal{H}_n(q)$ \cite[\S 4]{jo}:
$$
\mathcal{S}_\mathcal{H} = \{(G_{i_1}\ldots G_{i_1-r_1})(G_{i_2}\ldots G_{i_2-r_2})\cdots (G_{i_p}\ldots G_{i_p-r_p})\,|\,
1\leq i_1 <\cdots<i_p \leq n-1\}.
$$
Note that all generators $G_i$ are invertible in  $\mathcal{H}_n(q)$, with
\begin{equation}\label{hinv}
\begin{array}{rclcc}
G_i^{-1}&=&q^{-1}G_i+ (q^{-1}-1)&& \mbox{for all $i$}\,.
\end{array}
\end{equation}

\subsection{\it  Computation formulas in the Iwahori-Hecke algebra}

Let $m \in \mathbb{N}$. It is easy to check that
$$
G_i^m = \left( q^{m-1}-q^{m-2}+ \cdots +(-1)^{m-2}q+(-1)^{m-1}\right)G_i + \left(q^{m-1}-q^{m-2}+ \cdots +(-1)^{m-2}q \right).
$$
Hence, if $m$ is even, we have
\begin{equation}\label{hpowersev}
G_i^m = \frac{q^m-1}{q+1}\,G_i + \frac{q^m-1}{q+1} +1,
\end{equation}
and if $m$ is odd, we have
\begin{equation}\label{hpowersod}
G_i^m = \frac{q^m+1}{q+1}\, G_i + \frac{q^m+1}{q+1} -1 .
\end{equation}

\subsection{\it  The Ocneanu trace $\tau$}\label{wmu1}

The natural inclusions $B_n  \subset B_{n+1}$ of the classical  braid groups give rise to  the algebra inclusions:
$$
{{\mathbb C}}B_0  \, \subset {{\mathbb C}}B_1 \subset {{\mathbb C}}B_2 \subset \ldots
$$
(setting ${{\mathbb C}}B_0:={{\mathbb C}}$), which in turn induce the following algebra inclusions:
\begin{equation}
\mathcal{H}_0(q)  \, \subset \mathcal{H}_1(q) \subset \mathcal{H}_2(q) \subset \ldots
\end{equation}
(setting $\mathcal{H}_0(q):={{\mathbb C}}$). Then we have the following.

\begin{thm}\cite[Theorem 5.1]{jo}\label{otrace} 
Let $\zeta$ be an indeterminate. There exists a unique linear Markov trace
$$
\tau :  \bigcup_{n \geq 0} \mathcal{H}_{n}(q) \longrightarrow  {\mathbb C}[\zeta]
$$
defined inductively on $\mathcal{H}_n(q)$ for all $n$, by the following rules:
$$
\begin{array}{rcll}
\tau(hh') & = &\tau(h'h)  \qquad &  \\
\tau(1) & = & 1 & \\
\tau(hG_n) & = & \zeta \, \tau(h) \qquad &  (\text{Markov  property} )
\end{array}
$$
where $h,h' \in \mathcal{H}_n(q)$.
\end{thm}
The trace $\tau$ is the \emph{Ocneanu trace} with parameter $\zeta$. Diagrammatically, in the second rule, $1$ corresponds to the identity braid on any number of strands. The third rule is the so-called Markov property of the trace. One can look at the left-hand illustration of Figure~\ref{tracerules} for a topological interpretation of the Markov  property.

Another characterization of the Ocneanu trace can be given as follows. Every trace function is uniquely determined by its values on the basis elements $G_w$, where $w$ runs over a certain set of representatives of the various conjugacy classes of $\mathfrak{S}_n$ \cite[8.2.6]{gp}. Following \cite[3.1.16]{gp}, these particular representatives are the elements
$w_\mu$, defined in \S\ref{conj} for every partition $\mu$ of $n$. Applying the defining formula for the Ocneanu trace $\tau$ to the element $G_{w_\mu}$, we see that
\begin{equation}\label{ocneanu2}
\tau(G_{w_\mu})=\zeta^{\ell(w_\mu)}.
\end{equation}
Conversely, if $\psi$ is any trace function on $\mathcal{H}_n(q)$ such that $\psi(G_{w_\mu})=\zeta^{\ell(w_\mu)}$ for all partitions $\mu$ of $n$, then $\psi=\tau$.

\subsection{\it  The HOMFLYPT polynomial}\label{homfly-pt}

Let now ${\mathcal L}$ be the set of isotopy classes of oriented links in $S^3$. We know from Jones' construction \cite{jo} that, in order to obtain a link invariant according to the Markov equivalence for braids,  the closed braids $\widehat{\a}$, $\widehat{\a \sigma_n}$ and $\widehat{\a \sigma_n^{-1}}$ have to be assigned the same value for any braid $\a\in B_n$. Therefore, in order to obtain a link invariant from the trace $\tau$, it has to be re-scaled, so that
$\tau(h G_n^{-1}) = \tau(h G_n) \text{ for all } h\in \mathcal{H}_n(q)$, and also normalized, so that the closed braids $\widehat{\a}$ and $\widehat{\a \sigma_n} $ be assigned the same value. Set now
$$
\lambda_{\mathcal{H}}:=\frac{\zeta+(1-q)}{q\zeta}.
$$

\begin{defn}\cite[Definition 6.1] {jo}\label{homfly-pt def} \rm
 We define a map $P$ on the set ${\mathcal L}$ by defining $P$ on the closure $\widehat{\alpha}$ of any braid $\alpha\in B_n$, for all $n \in \mathbb{N}$, as follows:
$$
P (\widehat{\alpha}) := \left(- \frac{1-\lambda_{\mathcal{H}} q}{\sqrt{\lambda_{\mathcal{H}}}(1-q)}\right)^{n-1}
(\sqrt{\lambda_{\mathcal{H}}})^{\epsilon(\alpha)} \left(\tau \circ \pi \right)(\alpha)
$$
where $\pi: {\mathbb C}B_{n}  \longrightarrow  \mathcal{H}_{n}(q)$ is the natural algebra epimorphism that maps the braid generator $\sigma_i$ to the algebra generator $G_i$, and  $\epsilon(\alpha)$ is the sum of the exponents of the braid generators in the braid word $\alpha$. Equivalently, by setting
$$
D_{\mathcal{H}}:= - \frac{1-\lambda_{\mathcal{H}} q}{\sqrt{\lambda_{\mathcal{H}}}(1-q)} = \frac{1}{\zeta\sqrt{\lambda_{\mathcal{H}}}}
$$
we can write:
$$
P (\widehat{\alpha})
= {(D_{\mathcal{H}})}^{n-1}(\sqrt{\lambda_{\mathcal{H}}})^{\epsilon(\alpha)} \left(\tau \circ \pi \right)(\alpha).
$$
\end{defn}
As it turns out, the map $P$ is well-defined on ${\mathcal L}$ and it defines the well-known {\it 2-variable Jones} or {\it HOMFLYPT polynomial}, an isotopy invariant of classical knots and links.
This map depends on the quadratic relation (\ref{hecke})$\mathrm{(h)}$, so an automorphism of the Iwahori-Hecke algebra $\mathcal{H}_n(q)$ may give rise to a different map. However, one can easily check that the map $P'$ induced by the automorphism
\begin{equation}\label{heckeauto}
G_i \mapsto -q^{-1}G_i
\end{equation}
is equal to $P$ (if  $G_i':=-q^{-1}G_i$, then $G_i'^2=(q^{-1}-1)G_i'+q^{-1}$ and $\tau(G_i')=-q^{-1}\zeta$). We will need this later.

\section{Knot invariants from the Yokonuma-Hecke algebras}

\subsection{\it The Yokonuma-Hecke algebra}

In the sequel we fix  $d\in {\mathbb N}$. Let $u \in {\mathbb C}\backslash \{0\}$.  The Yokonuma-Hecke algebra, denoted by ${\rm Y}_{d,n}(u)$, is a ${\mathbb C}$-associative algebra generated by the elements
$$
 g_1, \ldots, g_{n-1}, t_1, \ldots, t_{n}
$$
subject to the following relations:
\begin{equation}\label{modular}
\begin{array}{ccrclcl}
\mathrm{(b}_1)& & g_ig_j & = & g_jg_i && \mbox{for $\vert i-j\vert > 1$}\\
\mathrm{(b}_2)& & g_ig_jg_i & = & g_jg_ig_j && \mbox{for $ \vert i-j\vert = 1$}\\
\mathrm{(f}_1)& & t_i t_j & =  &  t_j t_i &&  \mbox{for all $ i,j$}\\
\mathrm{(f}_2)& & t_j g_i & = & g_i t_{s_i(j)} && \mbox{for all $ i,j$}\\
\mathrm{(f}_3)& & t_j^d   & =  &  1 && \mbox{for all $j$}
\end{array}
\end{equation}
where $s_i$ is the transposition $(i, i+1)$, together with the extra quadratic relations:
\begin{equation}\label{quadr}
g_i^2 = 1 + (u-1) \, e_{i} + (u-1) \, e_{i} \, g_i \qquad \mbox{for all $i$}
\end{equation}
 where
\begin{equation}\label{ei}
e_i :=\frac{1}{d}\sum_{s=0}^{d-1}t_i^s t_{i+1}^{-s}.
\end{equation}

It is easily verified that the elements $e_i$ are idempotents in ${\rm Y}_{d,n}(u)$ \cite[Lemma 4]{ju}.  Also, that the elements $g_i$ are invertible, with
\begin{equation}\label{invrs}
g_i^{-1} = g_i + (u^{-1} - 1)\, e_i + (u^{-1} - 1)\, e_i\, g_i.
\end{equation}

The relations $({\rm b}_1)$, $({\rm b}_2)$, $({\rm f}_1)$ and $({\rm f}_2)$ are defining relations for the classical framed braid group $\mathcal{F}_n \cong \Z^n\rtimes B_n$, with the $t_j$'s being interpreted as the `elementary framings' (framing 1 on the $j$th strand). The relations $t_j^d = 1$ mean that the framing of each braid strand is regarded modulo~$d$.
Thus, the algebra ${\rm Y}_{d,n}(u)$ arises naturally  as a quotient of the framed braid group algebra over the modular relations ~$\mathrm{(f}_3)$ and the quadratic relations~(\ref{quadr}) \cite{ju}. Moreover, relations (\ref{modular}) are defining relations for the  modular framed braid group
$
\mathcal{F}_{d,n}\cong (\Z/d\Z)^n\rtimes B_n,
$
so the algebra ${\rm Y}_{d,n}(u)$ can be also seen  as a quotient of the modular framed braid group algebra over the  quadratic relations~(\ref{quadr}).

From the above, the algebra ${\rm Y}_{d,n}(u)$ has natural topological interpretation in the context of framed braids and framed knots. However, in~\cite{jula3}  a different topological interpretation to ${\rm Y}_{d,n}(u)$ was given, in relation to classical knots and links. Namely, viewing the $t_j$'s only as formal generators and ignoring their framing interpretation, we have by relations $({\rm b}_1)$ and  $({\rm b}_2)$ that the classical braid group $B_n$ is represented in ${\rm Y}_{d,n}(u)$.

The Yokonuma-Hecke algebra was originally introduced by T.~Yokonuma~\cite{yo}. For $d=1$, the algebra ${\rm Y}_{1,n}(u)$ coincides with the Iwahori-Hecke algebra $\mathcal{H}_n(u)$ of type $A$.
For more details  and for further topological interpretations, see \cite{jula1, jula2, jula3, jula4} and references therein.

\subsection{\it Computation formulas in the Yokonuma-Hecke algebra}

Let $i,k \in\{1,2,\ldots, n\}$ and let
\begin{equation}\label{edik}
e_{i,k} := \frac{1}{d} \sum_{s=0}^{d-1}t_i^st_{k}^{d-s}.
\end{equation}
 Clearly
$e_{i,k} = e_{k,i}$ and it can be easily deduced that $e_{i,k}^2 = e_{i,k}$. Note that $e_{i,i}=1$ and that $e_{i, i+1}=e_i$.
Now, in ${\rm Y}_{d,n}(u)$ the following relations hold (see \cite[Lemma~4, Proposition~5]{jula1}):
\begin{equation}\label{edikrels}
\begin{array}{rcll}
t_j e_{i} & = & e_{i} t_j & \\
e_{j} e_{i} & = & e_{i} e_{j} & \\
g_j e_{i} & = & e_{i} g_j & \text{ for } j\not= i-1, i+1 \\
g_{i-1} e_{i} & = & e_{i-1,i+1} g_{i-1}&\\
e_{i} g_{i-1}& =& g_{i-1} e_{i-1,i+1}&\\
g_{i+1} e_{i} & = & e_{i,i+2} g_{i+1}&\\
e_{i} g_{i+1}& =& g_{i+1} e_{i,i+2}& \\
e_j g_ig_j & = & g_ig_je_i & \text{ for } \vert i - j \vert =1.
\end{array}
\end{equation}
Note that, using (\ref{invrs}), relations (\ref{edikrels}) are also valid if all $g_k$'s are replaced by their inverses $g_k^{-1}$.  Moreover, the following relations hold in ${\rm Y}_{d,n}(u)$.

\begin{lem}\label{newequal}
Let $i,k \in\{1,2,\ldots, n\}$. We have
$$
t_ie_{i,k}=t_{k}e_{i,k}.
$$
In particular,
$$
t_ie_{i}=t_{i+1}e_{i}.
$$
\end{lem}
\begin{proof}
We have
$$t_ie_{i,k}=\frac{1}{d} \sum_{s=0}^{d-1}t_i^{s+1}t_{k}^{d-s}=\frac{1}{d} \sum_{r=1}^{d}t_i^{r}t_{k}^{d-r+1}=\frac{1}{d} \left(\sum_{r=1}^{d-1}t_i^{r}t_{k}^{d-r+1}+t_k\right)=
t_k\left(\frac{1}{d} \sum_{r=0}^{d-1}t_i^{r}t_{k}^{d-r}\right)=t_ke_{i,k}.
$$
\end{proof}

The following equalities are easy to check (see, for example, \cite[Lemma1]{jula2}):

\begin{lem}\label{powers}
Let $m \in \mathbb{N}$.  Then, if $m$ is even, we have
$$
g_i^{m} = \frac{u^m -1}{u+1} \, e_{i}\, g_i  + \frac{u^m -1}{u+1} \, e_{i} +  1,
$$
and if $m$ is odd, we have
$$
g_i^{m} = \frac{u^{m} -u}{u+1} \, e_{i} g_i + \frac{u^{m} -u}{u+1} \, e_{i} + g_i.
$$
\end{lem}

\subsection{\it The  Juyumaya trace ${\rm tr}$}

The natural inclusions ${\mathcal F}_n  \subset {\mathcal F}_{n+1}$ of the classical framed braid groups induce natural inclusions ${\mathcal F}_{d,n} \subset {\mathcal F}_{d,n+1}$ of modular framed braid groups and these give rise to the algebra inclusions:
$$
{\mathbb C}{\mathcal F}_{d,0}  \, \subset {\mathbb C}{\mathcal F}_{d,1} \subset \, {\mathbb C}{\mathcal F}_{d,2} \subset \ldots
$$
(setting ${\mathbb C}{\mathcal F}_{d,0}:={\mathbb C}$), which in turn induce the following algebra inclusions:
\begin{equation}\label{tower}
{\rm Y}_{d,0}(u)  \, \subset {\rm Y}_{d,1}(u) \subset {\rm Y}_{d,2}(u) \subset \ldots
\end{equation}
(setting ${\rm Y}_{d,0}(u):={\mathbb C}$). Then we have the following:

\begin{thm}\cite[Theorem 12]{ju}\label{trace} 
Let $z,\, x_1,\, \ldots, x_{d-1}$ be indeterminates. There exists a unique linear Markov trace
$$
{\rm tr}:  \bigcup_{n \geq 0}{\rm Y}_{d,n}(u) \longrightarrow   {\mathbb C}[z, x_1, \ldots, x_{d-1}]
$$
defined inductively on ${\rm Y}_{d,n}(u)$ for all $n$, by the following rules:
$$
\begin{array}{rcll}
{\rm tr}(ab) & = & {\rm tr}(ba)  \qquad &  \\
{\rm tr}(1) & = & 1 & \\
{\rm tr}(ag_n) & = & z\, {\rm tr}(a) \qquad &  (\text{Markov  property} )\\
{\rm tr}(at_{n+1}^m) & = & x_m {\rm tr}(a)\qquad  & (m = 1, \ldots , d-1)
\end{array}
$$
where $a,b \in {\rm Y}_{d,n}(u)$.
\end{thm}
We shall call the trace ${\rm tr}$ the \emph{Juyumaya trace}  with parameters $z,  x_1, \ldots,  x_{d-1}$. Diagrammatically, in the second rule, $1$ corresponds to the identity braid  on any number of strands with all framings zero.  The following figure gives the topological interpretations of the last two rules.

\begin{figure}[H]
\begin{picture}(330,70)
\put(-23,36){${\rm tr}$}
\qbezier(0,0)(-9,37)(0,74)  
\put(13,36){$a$}
\qbezier(8,0)(8,12)(8,24)
\qbezier(8,54)(8,62)(8,70)
\qbezier(22,54)(22,62)(22,70)
\qbezier(50,24)(52,47)(50,70)
\qbezier(30,0)(29,2)(30,4)
\qbezier(50,0)(51,2)(50,4)
\qbezier(58,0)(67,37)(58,74)  
\put(68,36){$=$}
\put(78,36){$z\,{\rm tr}$}
\qbezier(107,0)(96,37)(107,74)  
\qbezier(0,24)(0,39)(0,54)
\qbezier(30,24)(30,39)(30,54)

\qbezier(0,54)(15,54)(30,54)
\qbezier(0,24)(15,24)(30,24)
\qbezier(30,4)(30,9)(40,14)
\qbezier(40,14)(50,19)(50,24)
\qbezier(50,4)(50,8)(45,11)
\qbezier(35,17)(30,20)(30,24)
\qbezier(110,24)(110,39)(110,54)
\qbezier(140,24)(140,39)(140,54)

\qbezier(110,54)(125,54)(140,54)
\qbezier(110,24)(125,24)(140,24)
\qbezier(118,0)(118,12)(118,24)
\qbezier(118,54)(118,62)(118,70)
\qbezier(132,54)(132,62)(132,70)
\qbezier(132,0)(132,12)(132,24)
\put(122,36){$a$}
\qbezier(146,0)(156,37)(146,74) 
\put(160,36){$,$}
\put(175,36){${\rm tr}$}
\qbezier(197,0)(188,37)(197,74)  
\qbezier(202,24)(202,39)(202,54)
\qbezier(232,24)(231,39)(232,54)

\qbezier(202,54)(217,54)(232,54)
\qbezier(202,24)(217,24)(232,24)
\qbezier(210,0)(210,12)(210,24)
\qbezier(224,54)(224,62)(224,70)
\qbezier(210,54)(210,62)(210,70)
\qbezier(224,0)(224,12)(224,24)

\qbezier(240,0)(240,34)(240,68)
\put(238, 70){\tiny{ $m$}}
\put(215,36){$a$}
\qbezier(253,0)(261,37)(253,74) 
\put(261,36){$=$}
\put(272,36){$x_m\, {\rm tr}$}
\qbezier(306,0)(298,37)(306,74)  
\qbezier(310,24)(310,39)(310,54)
\qbezier(340,24)(340,39)(340,54)

\qbezier(310,54)(325,54)(340,54)
\qbezier(310,24)(325,24)(340,24)
\qbezier(318,0)(318,12)(318,24)
\qbezier(318,54)(318,62)(318,70)
\qbezier(332,54)(332,62)(332,70)
\qbezier(332,0)(332,12)(332,24)

\put(321,36){$a$}
\qbezier(345,0)(353,37)(345,74) 
\end{picture}
\caption{Topological interpretations of the trace rules}
\label{tracerules}
\end{figure}

The trace rules yield the following relations for all $i$:
\begin{equation}\label{traceofsquare}
{\rm tr}(e_i)=\frac{1}{d}\sum_{s=0}^{d-1}x_sx_{d-s}=:E \,\,\,\,\,\text{and}\,\,\,\,\, {\rm tr}(e_ig_i)={\rm tr}(g_i)=z.
\end{equation}
Using (\ref{traceofsquare}), Lemma \ref{powers}
implies that the following relations hold:
For $m \in \mathbb{Z}^{>0}$, we have
\begin{equation} \label{powersev}
{\rm tr}( g_i^m) = \frac{u^m-1}{u+1}\, z +\frac{u^m-1}{u+1}\, E+1
\quad \text{ \ if \ } m \text{ is even}
\end{equation}
and
\begin{equation} \label{powerseod}
{\rm tr}( g_i^m) =\frac{u^m+1}{u+1}\, z+\frac{u^m+1}{u+1} \,E - E
\quad \text{ \ if \ } m \text{ is odd}.
\end{equation}

\subsection{\it An inductive basis  for the Yokonuma-Hecke algebra}\label{linbasind}

The key in the construction of the trace ${\rm tr}$ is that ${\rm Y}_{d,n}(u)$ has a `nice' inductive linear $\mathbb{C}$-basis. Namely, every element of ${\rm Y}_{d,n+1}(u)$ is a unique linear combination of words of the following types:
\begin{equation}\label{basis}
 w_ng_ng_{n-1} \ldots g_i t_i^k \text{ \ \ \ \ \ or \ \ \ \ \ } w_nt_{n+1}^k \qquad  (k\in {\mathbb Z}/d{\mathbb Z})
\end{equation}
where $w_n \in {\rm Y}_{d,n}(u)$.  Thus, the above
words furnish an inductive basis for ${\rm Y}_{d,n+1}(u)$, every element of which involves  $g_n$ or a power of $t_{n+1}$ at most once.

\subsection{\it The split property for the Yokonuma-Hecke algebra}\label{linbasind2}

Due to the relations (\ref{modular})(${\rm f}_1$) and  (\ref{modular})(${\rm f}_2$), every monomial $w$ in ${\rm Y}_{d,n}(u)$ can be written in the form
$$
w=t_1^{k_1}\ldots t_n^{k_n}\cdot \sigma
$$
where $k_1,\ldots,k_n \in {\mathbb Z}/d{\mathbb Z}$ and $\sigma$ is a word in $g_1,\ldots,g_{n-1}$. That is, $w$ splits into the `framing part'
$t_1^{k_1}\ldots t_n^{k_n}$ and the `braiding part' $\sigma$. Applying further the braid relations
(\ref{modular})(${\rm b}_1$) and  (\ref{modular})(${\rm b}_2$) and the quadratic relations (\ref{quadr}), we deduce that the following set is a $\mathbb{C}$-basis for ${\rm Y}_{d,n}(u)$ \cite{ju, jula1}:
$$
\mathcal{S}_\mathrm{Y} =\left\{t_1^{k_1}\ldots t_n^{k_n}(g_{i_1}\ldots g_{i_1-r_1})(g_{i_2}\ldots g_{i_2-r_2})\cdots (g_{i_p}\ldots g_{i_p-r_p})\,\left|\,\begin{array}{l} k_1,\ldots,k_n \in {\mathbb Z}/d{\mathbb Z}\\
1\leq i_1 <\cdots<i_p \leq n-1\end{array}\right\}\right.
$$

\subsection{\it The {\rm E}-system}

Let now ${\mathcal L}$ be, as above, the set of isotopy classes of oriented links in $S^3$. As mentioned in Subsection~ \ref{homfly-pt}, likewise here, in order to obtain a link invariant from the trace ${\rm tr}$ according to the Markov equivalence for braids, the trace has to be: normalized, so that the closed braids $\widehat{\a}$ and $\widehat{\a \s_n}$ $(\a\in B_n)$ be assigned the same value of the invariant, and re-scaled, so that the closed braids $\widehat{\a \s_n^{-1}}$ and $\widehat{\a \s_n}$  $(\a\in B_n)$ be assigned the same value  of the invariant.

 Trying to do that, it was shown in \cite{jula2} that ${\rm tr}$ does not re-scale directly (being the only known Markov trace with this property). Indeed, for $\a\in {\rm Y}_{d,n}(u)$, we compute:
$$
{\rm tr}(\alpha g_n^{-1}) = {\rm tr}(\alpha g_n) + (u^{-1}-1){\rm tr}(\alpha e_{n})
 + (u^{-1}-1){\rm tr}(\alpha e_{n}g_n).
$$
Now, although
\begin{equation}\label{traceaengn}
{\rm tr}(\alpha e_{n}g_n) = {\rm tr}(\alpha g_n) = z\, {\rm tr}(\alpha) = {\rm tr}(g_n){\rm tr}(\alpha)
\end{equation}
we have that ${\rm tr}(\alpha e_{n})$ does not factor through ${\rm tr}(\alpha)$, that is,
$$
{\rm tr}(\alpha e_{n}) \neq {\rm tr}(e_{n}) {\rm tr}(\alpha)
$$
leading to the fact that ${\rm tr}(\alpha g_n^{-1})$ does not factor through ${\rm tr}(\alpha)$, that is,
$$
{\rm tr}(\alpha g_n^{-1}) \neq {\rm tr}(g_n^{-1}) {\rm tr}(\alpha).
$$
Forcing ${\rm tr}(\alpha e_{n}) = {\rm tr}(e_{n}) {\rm tr}(\alpha)$
yields that the trace parameters $x_1, \ldots, x_{d-1}$ have to satisfy the so-called ${\rm E}$-\emph{system} \cite[\S 4]{jula2}. The E-system is a non-linear system of equations of the form:

\begin{equation}\label{Esystem}
\sum_{s=0}^{d-1}{x}_{m + s} {x}_{d-s}  =  {x}_m \, \sum_{s=0}^{d-1} {x}_{s} {x}_{d-s} \qquad (1\leq m \leq d-1)
\end{equation}
where the sub-indices on the ${x}_j$'s are regarded modulo $d$ and ${x}_0:=1$. Equivalently, the E-system is written as
$$
E^{(m)} = {x}_m E
$$
where
$$
E^{(m)} :=\frac{1}{d} \sum_{s=0}^{d-1}{x}_{m+s}{x}_{d-s} \qquad \mbox{and} \qquad
E := E^{(0)} =\frac{1}{d}\sum_{s=0}^{d-1}{x}_{s}{x}_{d-s} = {\rm tr}(e_{i}).
$$
As it is shown in \cite{jula2} (in the Appendix by Paul G\'{e}rardin), the solutions of the E-system are parametrized by the non-empty subsets of ${\mathbb Z}/d{\mathbb Z}$. Let $X_{S}=\{{\rm x}_1, \ldots , {\rm x}_{d-1}\}$ be  a solution of the E-system parametrized by the non-empty subset $S$ of ${\mathbb Z}/d{\mathbb Z}$, and let us denote by ${\rm tr}_{_S}$ the Juyumaya trace with the parameters $x_1,\ldots,x_{d-1}$ specialized to ${\rm x}_1, \ldots , {\rm x}_{d-1}$. 
Then, as it turns out \cite{jula3},
\begin{equation}\label{valedi}
E = {\rm tr}_{_S} \left( e_{i}\right) = \frac{ 1}{\vert S\vert}\qquad (1\leq i\leq n-1).
\end{equation}
Moreover, we have \cite[Theorem 7]{jula2}:
\begin{equation} \label{etraceanen}
{\rm tr}_{_S}(\alpha e_{n}) \stackrel{\rm E}{=} {\rm tr}_{_S}(\alpha)\, {\rm tr}_{_S}( e_{n})=E\, {\rm tr}_{_S}(\alpha) \qquad (\a\in {\rm Y}_{d,n}(u)).
\end{equation}

\begin{notation} \rm
The symbol `$\stackrel{\rm E}{=}$' will stand for `$=$' up to the E-condition, that is, with a given solution of the E-system.
\end{notation}

\subsection{ \it The Case $E=1$}

The `trivial' solutions of the E-system are the ones parametrized by the singleton subsets of ${\mathbb Z}/d{\mathbb Z}$. By (\ref{valedi}), if $S$ is a singleton then $E= {\rm tr}_{_S} \left( e_{i}\right) =1$.
In this case,  G\'{e}rardin has shown, in the Appendix of  \cite{jula2}, that ${\rm x}_1$ is a $d$-th root of unity  and ${\rm x}_m={\rm x}_1^m$ $(1\leq m \leq d-1)$.
Consequently,
\begin{equation}\label{e1}
{\rm x}_{k+l}={\rm x}_1^{k+l}={\rm x}_1^k{\rm x}_1^l={\rm x}_k{\rm x}_l \qquad (k,l \in  {\mathbb Z}/d{\mathbb Z}).
\end{equation}
These solutions are not very interesting topologically, but we prove here that they have the following interesting property (a stronger version of  (\ref{etraceanen})):

\begin{prop}\label{E=1}
Let $X_S$ be a solution of the {\rm E}-system such that $E=1$. Then
$${\rm tr}_{_S}(\beta e_{j})\stackrel{\rm E}{=} {\rm tr}_{_S}(\beta)\, {\rm tr}_{_S}( e_{j})={\rm tr}_{_S}(\beta) \qquad (\beta\in {\rm Y}_{d,n+1}(u),\,\,1 \leq j \leq n).$$
\end{prop}
\begin{proof} It is enough to show that the above equality holds for the elements of the inductive basis given in (\ref{basis}). We will proceed by induction on $n$.
 Let $n=1$ and let $k,\,l \in {\mathbb Z}/d{\mathbb Z}$. We have
$$
{\rm tr}_{_S}(t_1^kg_1t_1^le_{1}) \stackrel{(\ref{edikrels})}{=} {\rm tr}_{_S}(t_1^kg_1e_{1}t_1^l) \stackrel{(\ref{traceaengn})}{=} {\rm tr}_{_S}(t_1^{k}g_1t_1^l).
$$
Moreover, following Lemma \ref{newequal},  we obtain:
$$
{\rm tr}_{_S}(t_1^{k}t_2^{l}e_1)={\rm tr}_{_S}(t_1^{k+l}e_1) \stackrel{\rm E}{=} {\rm tr}_{_S}(t_1^{k+l})\, {\rm tr}_{_S}( e_{1})= {\rm x}_{k+l}\stackrel{(\ref{e1})}{=}{\rm x}_{k}{\rm x}_{l}={\rm tr}_{_S}(t_1^{k}t_2^{l}).
$$
Now let $n>1$ and assume that the statement of the proposition holds for smaller values of $n$.

\bigbreak
Let $\beta=w_ng_ng_{n-1} \ldots g_i t_i^k$ for some $k\in {\mathbb Z}/d{\mathbb Z}$
and some $w_n \in {\rm Y}_{d,n}(u)$. Assume first that $j<n$.
Then
$$
{\rm tr}_{_S}(w_ng_ng_{n-1} \ldots g_i t_i^ke_j)=z {\rm tr}_{_S}(w_ng_{n-1} \ldots g_i t_i^ke_j).
$$
By the induction hypothesis, the latter is equal to
$$
z {\rm tr}_{_S}(w_ng_{n-1} \ldots g_i t_i^k)={\rm tr}_{_S}(w_ng_ng_{n-1} \ldots g_i t_i^k),
$$
so we are done. Now take $j=n$.
If $i=n$, then, by (\ref{edikrels}) and (\ref{traceaengn}):
$${\rm tr}_{_S}(w_ng_nt_n^ke_{n})={\rm tr}_{_S}(w_ng_nt_n^k).
$$
If $i<n$, then by (\ref{edikrels}):
$$
{\rm tr}_{_S}(w_ng_ng_{n-1} \ldots g_i t_i^ke_n)\stackrel{(\ref{modular})\mathrm{(f}_2)}{=}{\rm tr}_{_S}(w_ne_{n-1}g_ng_{n-1} \ldots g_i t_i^k)=z\, {\rm tr}_{_S}(g_{n-1} \ldots g_i t_i^kw_ne_{n-1}).
$$
By the induction hypothesis, the latter is equal to $z{\rm tr}_{_S}(g_{n-1} \ldots g_i t_i^kw_n)$, whence we deduce that
$$
{\rm tr}_{_S}(w_ng_ng_{n-1} \ldots g_i t_i^ke_n)={\rm tr}_{_S}(w_ng_ng_{n-1} \ldots g_i t_i^k).
$$

\bigbreak
Now let $\beta=w_nt_{n+1}^k$ for some $k\in {\mathbb Z}/d{\mathbb Z}$
and some $w_n \in {\rm Y}_{d,n}(u)$. Assume again first that $j<n$.  Applying the trace definition and the induction hypothesis we obtain:
$$
{\rm tr}_{_S}(w_nt_{n+1}^k e_{j})={\rm x}_k{\rm tr}_{_S}(w_ne_j)\stackrel{\rm E}{=}{\rm x}_k{\rm tr}_{_S}(w_n)={\rm tr}_{_S}(w_nt_{n+1}^k).
$$
Now take $j=n$.
With the use of Lemma \ref{newequal} and (\ref{etraceanen}), we obtain:
$$
{\rm tr}_{_S}(w_nt_{n+1}^k e_{n})={\rm tr}_{_S}(w_nt_{n}^k e_{n}) \stackrel{\rm E}{=} {\rm tr}_{_S}(w_nt_{n}^k) {\rm tr}_{_S}( e_{n})={\rm tr}_{_S}(w_nt_{n}^k).
$$
We need to show that, under the assumptions of the proposition:
$$
{\rm tr}_{_S}(w_nt_{n}^k) = {\rm tr}_{_S}(w_nt_{n+1}^k)={\rm x}_k  {\rm tr}_{_S}(w_n).
$$
Again, it is enough to show this for the elements of the inductive basis of
$ {\rm Y}_{d,n}(u)$. We have already shown it for $n=1$ and we will proceed by induction on $n$.
 Let $w_n=w_{n-1}g_{n-1}g_{n-2} \ldots g_i t_i^l$ for some $l\in {\mathbb Z}/d{\mathbb Z}$
and some $w_{n-1} \in {\rm Y}_{d,n-1}(u)$. Then
$$
{\rm tr}_{_S}(w_n t_n^k) = {\rm tr}_{_S}(w_{n-1}g_{n-1}g_{n-2} \ldots g_i t_i^lt_n^k)={\rm tr}_{_S}(w_{n-1}t_{n-1}^kg_{n-1}g_{n-2} \ldots g_i t_i^l)=z\,
{\rm tr}_{_S}(w_{n-1}t_{n-1}^kg_{n-2} \ldots g_i t_i^l).
$$
By the induction hypothesis, the latter is equal to
$$
z\, {\rm x}_k{\rm tr}_{_S}(w_{n-1}g_{n-2} \ldots g_i t_i^l)={\rm x}_k{\rm tr}_{_S}(w_{n-1}g_{n-1}g_{n-2} \ldots g_i t_i^l) = {\rm x}_k{\rm tr}_{_S}(w_n).
$$

Finally, let  $w_n=w_{n-1}t_n^l$ for some $l\in {\mathbb Z}/d{\mathbb Z}$
and some $w_{n-1} \in {\rm Y}_{d,n-1}(u)$. Then
$$
{\rm tr}_{_S}(w_nt_{n}^k) = {\rm tr}_{_S}(w_{n-1}t_n^lt_n^k)={\rm tr}_{_S}(w_{n-1}t_n^{l+k})={\rm x}_{k+l}{\rm tr}_{_S}(w_{n-1})\stackrel{(\ref{e1})}{=}{\rm x}_k{\rm x}_l{\rm tr}_{_S}(w_{n-1})={\rm x}_k{\rm tr}_{_S}(w_{n-1}t_n^l) = {\rm x}_k{\rm tr}_{_S}(w_n).
$$
\end{proof}

Note that there are also non-trivial solutions of the E-system, with more interesting topological interpretations.
For a detailed analysis see \cite{jula2}.

\subsection{\it Invariants for classical knots from the Juyumaya trace}

 Given a solution $X_{S}=\{{\rm x}_1, \ldots ,{\rm x}_{d-1}\}$ of the E-system, we wish to define a link isotopy invariant $\Delta_S$. Let $\alpha \in B_n$. The E-condition guarantees that
 $\Delta_{S}(\widehat{\alpha\sigma_n}) = \Delta_S(\widehat{\alpha\sigma_n^{- 1}})$. In order for $\Delta_S(\widehat{\alpha\sigma_n}) = \Delta_S(\widehat{\alpha})$ to hold, we need to normalize. Thus, by setting:
$$
 \lambda_\mathrm{Y}:= \frac{z +(1-u)E}{uz}
$$
 we define the following map on the set ${\mathcal L}$ of oriented classical link types.

\begin{defn}\cite[Definition 3]{jula3} \rm
For a solution $X_S$ of the E-system parametrized by the subset $S$ of  ${\mathbb Z}/d{\mathbb Z}$, we define a map $\Delta_S$  on the set ${\mathcal L}$  by defining $\Delta_S$ on the closure $\widehat{\alpha}$ of any braid $\alpha\in B_n$, for all $n \in {\mathbb N}$, as follows:
$$
\Delta_S (\widehat{\alpha}) := \left(- \frac{1- \lambda_\mathrm{Y} u}{\sqrt{ \lambda_\mathrm{Y}}(1-u)E}\right)^{n-1}(\sqrt{\lambda_\mathrm{Y}})^{\epsilon(\alpha)}
\left({\rm tr}_{_S} \circ \delta\right)(\alpha)
$$
where $\delta : {\mathbb C}B_{n}  \longrightarrow  {\rm Y}_{d,n}(u)$ is the natural algebra homomorphism that maps the braid generator $\sigma_i$ to the algebra generator $g_i$, and  $\epsilon(\alpha)$ is the sum of the exponents of the braid generators in the braid word $\alpha$.
Equivalently, by setting
$$
D_\mathrm{Y}:=- \frac{1- \lambda_\mathrm{Y} u}{\sqrt{ \lambda_\mathrm{Y}}(1-u)E}=\frac{1}{z\sqrt{ \lambda_\mathrm{Y}}}
$$
 we can write:
$$
\Delta_S (\widehat{\alpha}) = D_\mathrm{Y}^{n-1}(\sqrt{\lambda_\mathrm{Y}})^{\epsilon(\alpha)} \left({\rm tr}_{_S} \circ \delta\right)(\alpha).
$$
\end{defn}

In \cite{jula3} the following result was obtained:

\begin{thm}\label{invariant}
For a solution $X_{S}$ of the {\rm E}-system, the map $\Delta_S$ is well-defined on the set ${\mathcal L}$, that is, it is a $2$-variable isotopy invariant for oriented classical links, depending on the variables $u$, $z$.
\end{thm}

Note that for every $d\in \mathbb{N}$, the above construction provides us with $2^d -1$ isotopy invariants for knots.

As shown in \cite{CJJKL}, the invariants $\Delta_S$ (but not the trace ${\rm tr}$) have the multiplicative property on connected sums of links.  Further, it was shown in \cite{jula3} that $\Delta_S$ satisfies a `closed' cubic relation (closed in the sense of only involving braiding generators), which factors to the quadratic relation of the Iwahori--Hecke algebra $\mathcal{H}_n(u)$.  Finally, it was shown in \cite{jula2} that in  the context of framed links the invariant $\Delta_S$ satisfies a skein relation. However, when  $\Delta_S$ is seen as invariant of classical knots, this skein relation has no topological interpretation. This makes it very difficult to compare the invariants  $\Delta_S$ with the HOMFLYPT polynomial using diagrammatic methods.

\section{The Ocneanu trace vs the specialized Juyumaya trace}

In order to compare the knot invariants $P$ and $\Delta_S$, in view of the {\rm E}-condition, we would like to be able to specialize the indeterminates $x_1, \ldots, x_{d-1}$ to a solution of the {\rm E}-system as early as possible during the construction.

\subsection{\it The algebra homomorphism approach}\label{alghom}
Our first approach to the above problem is to construct an algebra homomorphism $f:\bigcup_{n \geq 0}{\rm Y}_{d,n}(u) \longrightarrow \bigcup_{n \geq 0}{\rm Y}_{d,n}(u)$ such that
$$f(t_i^m)= {\rm x}_m \quad (1 \leq m \leq d-1),$$
where ${\rm x}_1,{\rm x}_2,\ldots,{\rm x}_{d-1} \in \mathbb{C} \setminus \{0\}$.
Since $f$ is an algebra homomorphism, we must have
$$f(t_i^kt_i^l)=f(t_i^{k+l})={\rm x}_{k+l} ={\rm x}_{k}{\rm x}_{l}=f(t_i^k)f(t_i^l)\,\,\,\,\text{for}\, k,l \in \mathbb{Z}/d\mathbb{Z},$$
that is, $ {\rm x}_1$ is a $d$-th root of unity and ${\rm x}_m={\rm x}_1^m$ $(1\leq m \leq d-1)$. In this case, the set $X_S=\{{\rm x}_1,\ldots, {\rm x}_{d-1}\}$ is a solution of the
{\rm E}-system such that $E=1$.
Moreover, we have
$$f(e_i)=\frac{1}{d}\sum_{s=0}^{d-1}f(t_i^s)f( t_{i+1}^{-s})=\frac{1}{d}\sum_{s=0}^{d-1}{\rm x}_s {\rm x}_{d-s}={\rm tr}_{_S}(e_i)=E=1$$
and
$$f(g_i)^2=f(g_i^2)=f\left(1 + (u-1) \, e_{i} + (u-1) \, e_{i} \, g_i\right)=1+(u-1)+(u-1)f(g_i)=u+(u-1)f(g_i),$$
whence we can easily deduce that $f({\rm Y}_{d,n}(u))$ is isomorphic to the Iwahori--Hecke algebra $\mathcal{H}_n(u)$. 

The condition $E=1$ is quite restrictive,  thus making $f$ an uninteresting mapping for our purposes. This  is why in the rest of this section we will proceed with a step-by-step specialization $t_i^m \mapsto {\rm x}_m$, namely with the construction of a specialized Juyumaya trace, and a linear map $\varphi$ on the Yokonuma-Hecke algebra through which the trace factors. This will allow us to conclude, in Subsection \ref{consequencesE=1}, that the invariants $P$ and $\Delta_S$ coincide when $E=1$.


\subsection{\it The specialized Juyumaya trace}\label{speju}

\begin{defn}\label{theta} \rm
Let ${\rm x}_1,{\rm x}_2,\ldots,{\rm x}_{d-1} \in \mathbb{C} \setminus \{0\}$
 and consider the ring homomorphism
$$
\begin{array}{rccc}
\theta :  \ {\mathbb C}[z, x_1, \ldots, x_{d-1}]  & \longrightarrow &  {\mathbb C}[z] & \\
 z   &  \mapsto & z &  \\
 x_m   & \mapsto & {\rm x}_m &(1 \leq m \leq d-1)
\end{array}
$$
The map $\theta$ shall be called the {\em specialization map}. We will call the composition
$$
\theta \circ {\rm tr} :  \bigcup_{n \geq 0}{\rm Y}_{d,n}(u) \longrightarrow {\mathbb C}[z]
$$
the \emph{specialized Juyumaya trace}  with parameter $z$.
\end{defn}

In the case where $X_S=\{{\rm x}_1,\ldots, {\rm x}_{d-1}\}$  is a solution  of the {\rm E}-system, we have $\theta \circ  {\rm tr}={\rm tr}_{_S}$. Note also that in the case $d=1$, when the algebra ${\rm Y}_{1,n}(u)$ coincides with the Iwahori-Hecke algebra $\mathcal{H}_n(u)$, $\theta$ is simply the identity map on $ {\mathbb C}[z]$ and the specialized Juyumaya trace $\theta \circ  {\rm tr}={\rm tr}$ coincides with the Ocneanu trace with parameter $z$.

\subsection{\it Similarities with the Ocneanu trace}\label{similarities1}

In this subsection we will give another characterization of the specialized Juyumaya trace as follows. Let $w\in \mathfrak{S}_n$ and let  $w=s_{i_1}s_{i_2}\ldots s_{i_r}$ be a reduced expression. Then we can set
$g_w:=g_{i_1}g_{i_2}\ldots g_{i_r}$. If $w,w' \in  \mathfrak{S}_n$ are such that $\ell(ww')=\ell(w)+\ell(w')$, then we have
\begin{equation}\label{multilength}
g_wg_{w'}=g_{ww'}.
\end{equation}

 Let $\mu$ be a partition of $n$ and let $w_\mu$ be the corresponding element of $\mathfrak{S}_n$ defined in \S\ref{conj}.
Applying the defining formula for the Juyumaya trace ${\rm tr}$ to the element $g_{w_\mu}$, we see that ${\rm tr}(g_{w_\mu})=z^{\ell(w_\mu)}$, whence we deduce:
\begin{equation}\label{spju2}
(\theta \circ {\rm tr})(g_{w_\mu})=z^{\ell(w_\mu)}.
\end{equation}
We will show that, as in the case of Iwahori-Hecke algebra of type $A$, the specialized Juyumaya trace  on  ${\rm Y}_{d,n}(u)$  is uniquely determined by its values on the elements
$g_{w_\mu}$, where $\mu$ runs over the partitions of $n$. That is, if $\psi$ is any trace function on ${\rm Y}_{d,n}(u)$ such that $\psi(g_{w_\mu})=z^{\ell(w_\mu)}$ for all partitions $\mu$ of $n$, then $\theta \circ \psi=\theta \circ {\rm tr}$.
 To achieve our aim, we shall first construct a linear map
$\varphi: \bigcup_{n \geq 0}{\rm Y}_{d,n}(u) \longrightarrow \bigcup_{n \geq 0}{\rm Y}_{d,n}(u)$
with the property:
${\rm tr} \circ \varphi = \theta \circ {\rm tr}$.

\begin{prop}\label{phi} Let $\theta$ be as above and set ${\rm x}_0:=1$.
Let
$
\varphi : \bigcup_{n \geq 0}{\rm Y}_{d,n}(u) \longrightarrow \bigcup_{n \geq 0}{\rm Y}_{d,n}(u)
$
be the linear map defined inductively on ${\rm Y}_{d,n}(u)$, for all $n \in \mathbb{N}$, by the following rules:
$$
\begin{array}{rcll}
\varphi(1) & = & 1  \qquad &  \\
 \varphi(w_ng_ng_{n-1} \ldots g_i t_i^k) & = &g_n \varphi(w_ng_{n-1} \ldots g_i t_i^k) & \\
\varphi(w_nt_{n+1}^k)& = & {\rm x}_k \varphi(w_n)
\end{array}
 $$
where $w_n \in {\rm Y}_{d,n}(u)$ and $k\in {\mathbb Z}/d{\mathbb Z}$. Then we have:
\begin{equation}\label{factorthrough}
{\rm tr} \circ \varphi = \theta \circ {\rm tr}.
\end{equation}
\end{prop}

\begin{proof}
It is enough to show that (\ref{factorthrough}) holds on the elements of the inductive basis of ${\rm Y}_{d,n+1}(u)$, and we will do this by induction on $n$. From now on, $k$ and $l$ are elements of  ${\mathbb Z}/d{\mathbb Z}$.

First, let  $n=1$. We have
$$
{\rm tr}\left(\varphi (t_1^kg_1) \right) ={\rm tr} \left(g_1 \varphi (t_1^k)\right) ={\rm tr}({\rm x}_kg_1)={\rm x}_kz = \theta(x_kz)= \theta\left({\rm tr} (t_1^kg_1) \right)
$$
and
$$
{\rm tr} \left(\varphi (t_1^lt_2^k)\right) ={\rm tr} \left({\rm x}_k \varphi (t_1^l)\right) ={\rm tr}({\rm x}_k{\rm x}_l)={\rm x}_k{\rm x}_l=\theta(x_kx_l)=\theta\left({\rm tr}(t_1^lt_2^k)\right).
$$
Now assume that (\ref{factorthrough}) holds for smaller values of $n$.
We have
$$
{\rm tr} \left(\varphi(w_ng_ng_{n-1} \ldots g_i t_i^k) \right) = {\rm tr} \left(g_n \varphi(w_ng_{n-1} \ldots g_i t_i^k) \right) = z\, {\rm tr} \left(\varphi(w_ng_{n-1} \ldots g_i t_i^k) \right),
$$
since $\varphi(w_ng_{n-1} \ldots g_i t_i^k) \in {\rm Y}_{d,n}(u)$.
By the induction hypothesis, the last term is equal to
$$
z\, \theta \left({\rm tr} (w_ng_{n-1} \ldots g_i t_i^k) \right) =\theta \left(z\,{\rm tr} (w_ng_{n-1} \ldots g_i t_i^k) \right) =
\theta \left({\rm tr} (w_ng_ng_{n-1} \ldots g_i t_i^k) \right).
$$
Finally, we have
 $$
 {\rm tr} \left(\varphi(w_nt_{n+1}^k) \right) ={\rm tr} \left({\rm x}_k \varphi(w_n) \right) ={\rm x}_k\,{\rm tr} \left( \varphi(w_n) \right).
$$
 By the induction hypothesis, the last term is equal to
$$
{\rm x}_k \theta \left({\rm tr} (w_n) \right) =\theta \left(x_k\,{\rm tr}(w_n) \right) = \theta({\rm tr} \left(w_nt_{n+1}^k) \right).
$$\end{proof}

\begin{rem}\label{ocnphi}\rm
The map $\varphi$ is an intermediate construction between the algebra and the trace map.
An analogue of the map $\varphi$ can be constructed on the Iwahori-Hecke algebra $\mathcal{H}_n(q)$.
\end{rem}

\begin{rem}\label{remphi}\rm
By virtue of Proposition~\ref{phi}, for a solution $X_{S} =\{{\rm x}_1,\ldots, {\rm x}_{d-1}\}$ of the {\rm E}-system, the invariant $\Delta_S$ is rewritten as:
$$
\Delta_S (\widehat{\alpha}) = D_\mathrm{Y}^{n-1}(\sqrt{\lambda_\mathrm{Y}})^{\epsilon(\alpha)} \left({\rm tr} \circ \varphi \circ \delta\right)(\alpha).
$$
Moreover, in view of the discussion in Subsection \ref{alghom}, and since $\varphi(g_i)=g_i$ and $\varphi(t_i^m)={\rm x}_m$,
 the map $\varphi$ of Proposition~\ref{phi} provides us with the earliest possible specialization of
$x_1, \ldots, x_{d-1}$ to $X_{S}$ during the construction of $\Delta_S$.
\end{rem}

Proposition~\ref{phi} implies that the specialized Juyumaya trace is uniquely determined by its values on the elements of the image of $\varphi$.
We will now show that  $\varphi({\rm Y}_{d,n}(u))$ is the subspace $W_n$ of ${\rm Y}_{d,n}(u)$ spanned by the elements
$\{g_{w}\}_{w \in \mathfrak{D}}$, where
$$\mathfrak{D}=\{s_{i_k}\ldots s_{i_2}s_{i_1} \,|\, i_1<i_2<\cdots<i_k\} \subset \mathfrak{S}_n.$$

\begin{prop}
Let $n \in \mathbb{N}$ and let $W_n$ be the ${\mathbb C}$-linear subspace of ${\rm Y}_{d,n}(u)$ spanned by the elements $\{g_{w}\}_{w \in \mathfrak{D}}$. Then $\varphi ( {\rm Y}_{d,n}(u)) = W_n$.
\end{prop}

\begin{proof}
First  note that we have $W_n \subset W_{n+1}$ and $g_nW_n \subset W_{n+1}$.
Note also that $\varphi(1)=1 \in W_n$.

We will first show that $\varphi( {\rm Y}_{d,n}(u) )\subseteq W_n$. We will proceed by induction on $n$.
Let $n=1$. Then $\varphi(t_1^k)={\rm x}_k \cdot 1 \in W_1$, for all $k \in  \mathbb{Z}/d\mathbb{Z}$.
Now assume that $\varphi( {\rm Y}_{d,n}(u) )\subseteq W_n$. In order to show that $\varphi( {\rm Y}_{d,n+1}(u) )\subseteq W_{n+1}$, it is enough to show that the images of the elements of the inductive basis of ${\rm Y}_{d,n+1}(u)$ under $\varphi$ are contained in $W_{n+1}$.
Let $w_{n} \in {\rm Y}_{d,n}(u)$ and  $k \in  \mathbb{Z}/d\mathbb{Z}$. We have
$$
\varphi(w_ng_ng_{n-1} \ldots g_i t_i^k)  = g_n \varphi(w_ng_{n-1} \ldots g_i t_i^k) \in g_nW_n \subset W_{n+1}
$$
and
$$
\varphi(w_nt_{n+1}^k) =  {\rm x}_k \varphi(w_n) \in W_n \subset W_{n+1}.
$$
We conclude that $\varphi( {\rm Y}_{d,n+1}(u) )\subseteq W_{n+1}$.

On the other hand, let $w \in \mathfrak{D}$. Then $g_{w}=\varphi(g_{w})$, and so $W_n \subseteq \varphi( {\rm Y}_{d,n}(u) )$.
\end{proof}

Now let $w\in \mathfrak{D}$. Following the discussion in \S\ref{conj}, $w$ has minimal length in its conjugacy class in $\mathfrak{S}_n$. Suppose that the conjugacy class of $w$ is parametrized by the partition $\mu$ of $n$.  By {Theorem~\ref{Cmin},}  the elements $w$ and $w_\mu$ are strongly conjugate, that is,
there exists a finite sequence $w=w_0,w_1,\ldots,w_r=w_\mu$ such that, for all $i=0,1,\ldots,r-1$,
$$\ell(w_{i})=\ell(w_{i+1}),\,\,\, w_{i+1}=x_iw_{i}x_i^{-1} \,\, \text{ and }\,\,  \ell(x_iw_i)=\ell(x_i)+\ell(w_i)\,\, \text{ or }\,\, \ell(w_ix_i^{-1})=\ell(w_i)+\ell(x_i^{-1})$$
for some elements $x_i \in \mathfrak{S}_n$.
If $\ell(x_iw_i)=\ell(x_i)+\ell(w_i)$, then we have 
$$\ell(w_{i+1}x_i)=\ell(x_iw_i)=\ell(x_i)+\ell(w_i)=\ell(w_{i+1})+\ell(x_i).$$
Following (\ref{multilength}), since $ w_{i+1}x_i=x_iw_{i}$ we deduce that
$$g_{w_{i+1}}g_{x_i}=g_{x_i}g_{w_{i}},$$
and thus,
$$g_{w_{i+1}}=g_{x_i}g_{w_{i}}g_{x_i}^{-1}.$$
On the other hand, if $\ell(w_ix_i^{-1})=\ell(w_i)+\ell(x_i^{-1})$, then we have
$$\ell(x_i^{-1}w_{i+1})=\ell(w_ix_i^{-1})=\ell(w_i)+\ell(x_i^{-1})=\ell(x_i^{-1})+\ell(w_{i+1}).$$
Following again (\ref{multilength}), since $x_i^{-1}w_{i+1}=w_{i}x_i^{-1}$,  we deduce that
$$g_{x_i^{-1}}g_{w_{i+1}}=g_{w_{i}}g_{x_i^{-1}},$$
and thus,
$$g_{w_{i+1}}=g_{x_i^{-1}}^{-1}g_{w_{i}}g_{x_i^{-1}}.$$
In any case, the elements $g_{w_{i+1}}$ and $g_{w_{i}}$ are conjugate in ${\rm Y}_{d,n}(u)$, and so, if $\psi$ is any trace function on ${\rm Y}_{d,n}(u)$, we have
$$\psi\left (g_{w_{i+1}} \right)=
\psi\left (g_{w_{i}} \right)$$
for all $i=1,\ldots,r$.
Thus, we have
$${\psi}\left (g_{w} \right)=
{\psi}\left (g_{w_{\mu}} \right).$$
In particular, we obtain that
$${\rm tr}\left (g_{w} \right)=
{\rm tr}\left (g_{w_{\mu}} \right).$$
From the above, we conclude the following:

\begin{prop}\label{simwitocn}
The specialized Juyumaya trace $\theta \circ  {\rm tr}$ on  ${\rm Y}_{d,n}(u)$  is uniquely determined by its values on the elements
$g_{w_\mu}$, where $\mu$ runs over the partitions of $n$.
\end{prop}

\subsection{\it Consequences on the case $E=1$}\label{consequencesE=1}

Suppose now that $X_S=\{{\rm x}_1,\ldots, {\rm x}_{d-1}\}$ is a solution of the {\rm E}-system such that $E=1$, that is,
${\rm x}_1$ is a $d$-th root of unity and ${\rm x}_m={\rm x}_1^m$ $(1\leq m \leq d-1)$. In this case, we can define the algebra epimorphism
$$
\begin{array}{rccc}
\gamma :   \ {\rm Y}_{d,n}(u) & \longrightarrow & \mathcal{H}_n(u) & \\
 g_i  &  \mapsto & G_i &  \\
 t_i^m   & \mapsto & {\rm x}_m &(1 \leq m \leq d-1).
\end{array}
$$
This is indeed an algebra homomorphism, since it respects all the defining relations of the algebra ${\rm Y}_{d,n}(u)$. In particular:
$$
\gamma(g_i^2)=\gamma\left(1 + (u-1) \, e_{i} + (u-1) \, e_{i} \, g_i\right)=u+(u-1)G_i=G_i^2=\gamma(g_i)^2
$$
and
$$\gamma(t_i^kt_i^l)=\gamma(t_i^{k+l})={\rm x}_{k+l} \stackrel{\rm E}{=}{\rm x}_{k}{\rm x}_{l}=\gamma(t_i^k)\gamma(t_i^l)\,\,\,\,\text{for}\, k,l \in \mathbb{Z}/d\mathbb{Z}.$$
The map $\gamma$ is also clearly surjective.

\begin{rem}
{\rm Note that the map $\gamma$ is not an algebra homomorphism if $E \neq 1$, because
neither of the above equalities holds.}
\end{rem}

Now consider the Ocneanu trace on $\mathcal{H}_n(u)$ with parameter $\zeta=z$. The composition $\tau \circ \gamma$ is a Markov trace on
${\rm Y}_{d,n}(u)$ which takes the same values as the specialized Juyumaya trace ${\rm tr}_{_S}$ on the elements $g_{w_\mu}$, where $\mu$ runs over the partitions of $n$. By Proposition \ref{simwitocn}, we obtain
\begin{equation}\label{firstcon}
 \tau \circ \gamma =  {\rm tr}_{_S}.
\end{equation}
The following result is a consequence of (\ref{firstcon}).

\begin{prop}\label{firstconprop}
Let $X_S$ be a solution of the {\rm E}-system such that $E=1$. Let ${\rm tr}_{_S}$ be the corresponding specialized Juyumaya trace on ${\rm Y}_{d,n}(u)$ with parameter $z$, and let
$\tau$ be the Ocneanu trace on $\mathcal{H}_{n}(q)$ with parameter $\zeta$. If we take $u=q$ and $z = \zeta$, then
$$ (\tau \circ \pi)(\alpha)= ( {\rm tr}_{_S} \circ \delta)(\alpha) \quad{(\a \in B_n)}$$
for all $n \in \mathbb{N}$.
\end{prop}

\begin{proof} Let $\alpha \in B_n$. By definition of the map $\gamma$, we have $(\gamma \circ \delta)(\a)=\pi(\a)$.
 Now Equation (\ref{firstcon}) yields the desired result.
\end{proof}

Under the assumptions of Proposition \ref{firstconprop}, we automatically obtain $\lambda_\mathcal{H}=\lambda_{\rm Y}$. We conclude the following.
\begin{cor}\label{firstcompinv}
Let $X_S$ be a solution of the {\rm E}-system such that $E=1$. Let ${\rm tr}_{_S}$ be the corresponding specialized Juyumaya trace on ${\rm Y}_{d,n}(u)$ with parameter $z$, and let
$\tau$ be the Ocneanu trace on $\mathcal{H}_{n}(q)$ with parameter $\zeta$. If we take $q=u$ and $\zeta=z$, then
$$ P(\widehat{\alpha})= \Delta_S(\widehat{\alpha}) \quad{(\a \in B_n)}$$
for all $n \in \mathbb{N}$.
\end{cor}

Since the map $P$ is invariant under the Hecke algebra automorphism (\ref{heckeauto}), we also obtain the following.
\begin{cor}\label{secondcompinv}
Let $X_S$ be a solution of the {\rm E}-system such that $E=1$. Let ${\rm tr}_{_S}$ be the corresponding specialized Juyumaya trace on ${\rm Y}_{d,n}(u)$ with parameter $z$, and let
$\tau$ be the Ocneanu trace on $\mathcal{H}_{n}(q)$ with parameter $\zeta$. If we take $q=1/u$ and $ \zeta=-z/u$, then
$$ P(\widehat{\alpha})= \Delta_S(\widehat{\alpha}) \quad{(\a \in B_n)}$$
for all $n \in \mathbb{N}$.
\end{cor}

In the next sections we will explore the remaining cases where the maps $P$ and $\Delta_S$ coincide, and show that they are all trivial, that is, either $u=1$ or $q=1$ or $E=1$.

\section{Comparing  $P$ and $\Delta_S$}

From now on, let $X_S=\{{\rm x}_1,\ldots,{\rm x}_{d-1}\}$ be a solution of the E-system parametrized by a subset $S$ of $\mathbb{Z}/d\mathbb{Z}$.
Let ${\rm tr}_{_S}$ be the corresponding specialized Juyumaya trace on ${\rm Y}_{d,n}(u)$ with parameter $z$, and let
$E={\rm tr}_{_S}(e_i)=1/|S|$. Let $\tau$ be the Ocneanu trace on $\mathcal{H}_{n}(q)$ with parameter $\zeta$.
In this section, we will assume that the maps $P$ and $\Delta_S$ coincide, and we will see what restrictions this assumption imposes on the values that $q$, $\zeta$, $u$, $z$ and $E$ can take.

\subsection{\it Some equalities}\label{Some equalities}

First of all, if $P$ and $\Delta_S$ coincide, they should take the same value on the closure of any braid in any $B_n$. In particular for the identity braid $1$ in each $B_n$ we have:
$$
P(\,\widehat{1}\,)=D_{\mathcal{H}}^{n-1}=D_{\rm Y}^{n-1}=\Delta_S(\,\widehat{1}\,)
$$
for all $n \in \mathbb{N}$, whence we deduce that
\begin{equation}\label{dprimed}
D_{\mathcal{H}}=D_{\rm Y}.
\end{equation}
From~(\ref{dprimed}) we obtain the following equality:
\begin{equation}\label{(1a)}
{(u\zeta+z^2-uEz+Ez)} \, q={u\zeta(\zeta+1)}.
\end{equation}
If $\zeta=-1$, then we must have $u\zeta+z^2-uEz+Ez=0$. If $\zeta \neq -1$, then (\ref{(1a)}) yields the following equality for $q$:
\begin{equation}\label{(1)}
q=\frac{u\zeta^2+u\zeta}{u\zeta+z^2-uEz+Ez}\,\,.
\end{equation}

Now if  $P=\Delta_S$ and $D_{\mathcal{H}}=D_{\rm Y}$, we must have
\begin{equation}\label{ratio}
\frac{\tau(\pi(\alpha))}{{\rm tr}_{_S}(\delta(\alpha))}=\left(\sqrt{\frac{\lambda_{\rm Y}}{\lambda_{\mathcal{H}}}}\right)^{\epsilon(\alpha)}=\left(\frac{\zeta}{z}\right)^{\epsilon(\alpha)}
\end{equation}
for all $\alpha\in B_n$ and for all $n \in \mathbb{N}$.
Taking $\alpha=\sigma_1^2 \in B_n$ and  $\a=\sigma_1^3 \in B_n$, for $n\geq 2$, we obtain respectively:
\begin{equation}\label{(2a)}
{(z^2\zeta+z^2)} \, q=\zeta({b\zeta+z^2}),
\end{equation}
where
$$b:={\rm tr}_{_S}(g_1^2)=1+(u-1)E+(u-1)z,$$
and
\begin{equation}\label{(3a)}
{(bz\zeta+z^3)}\, q=\zeta({c\zeta+bz}),
\end{equation}
where
$$c:={\rm tr}_{_S}(g_1^3)=(u^2-u)E+(u^2-u+1)z.$$
Note that Equation \ref{(2a)} implies that $\zeta=-1$ if and only if $b\zeta+z^2=0$.

From now on, let us assume that $\zeta \neq -1$. Then
(\ref{(2a)}) and (\ref{(3a)}) yield respectively the following equalities for $q$:
\begin{equation}\label{(2)}
q=\frac{b\zeta^2+z^2\zeta}{z^2\zeta+z^2}
\end{equation}
and
\begin{equation}\label{(3)}
q=\frac{c\zeta^2+bz\zeta}{bz\zeta+z^3}\,\,.
\end{equation}
Suppose first that $z^2 \neq b$. Combining (\ref{(1)}) and (\ref{(2)}) yields:
\begin{equation}\label{(5)}
\zeta^2=\frac{bz^2-ubEz+bEz-uz^2}{uz^2-ub} \zeta + \frac{z^4-uEz^3+Ez^3-uz^2}{uz^2-ub}.
\end{equation}
Suppose now that $bz \neq c$. Combining (\ref{(1)}) and (\ref{(3)}) yields:
\begin{equation}\label{(7)}
\zeta^2=\frac{cz^2-ucEz+cEz-uz^3}{ubz-uc} \zeta + \frac{bz^3-ubEz^2+bEz^2-uz^3}{ubz-uc}.
\end{equation}
Combining (\ref{(5)}) and (\ref{(7)}) yields $\zeta=-1$, which contradicts our assumption, unless
$$u=1 \,\,\,\text{or}\,\,\, E=1 \,\,\,\text{or}\,\,\,z=\frac{1-E+u+uE}{2}.$$
We conclude that the only cases where the invariants  $P$ and $\Delta_S$ may coincide are the following:
\begin{enumerate}[(1)]
\item $\zeta=-1$;
\item $\zeta \neq -1$, $z^2=b$\,;
\item $\zeta \neq -1$, $bz=c$\,;
\item $\zeta \neq -1$, $u=1$\,;
\item $\zeta \neq -1$, $E=1$\,;
\item  $\zeta \neq -1$, $z=(1-E+u+uE)/2$\,.
\end{enumerate}

\subsection{\it The case $\zeta = -1$}

If $\zeta=-1$, then Equations \ref{(1a)} and \ref{(2a)} imply that
$$z^2=uEz-Ez+u \,\,\,\text{and}\,\,\,z^2=b=1+(u-1)E+(u-1)z.$$
Combining the two equalities above yields
$$(u-1)(E+z)=(u-1)(Ez+1)$$
which is true only if $u=1$ or $E=1$ or $z=1$.

If $u=1$, then $z=-1$ or $z=1$. If $E=1$, then $z=-1$ or $z=u$. If $z=1$ and $u \neq 1$, then $E=-1$ which is absurd.

\subsection{\it The case $\zeta \neq -1$, $z^2=b$ }\label{z^2=b}

If $z^2=b=1+(u-1)E+(u-1)z$, then Equation~\ref{(2)} yields $q=\zeta$.
Replacing $q=\zeta$ in (\ref{(1)}), we obtain that $z^2=uEz-Ez+u$. As in the previous subsection, we conclude that $u=1$ or $E=1$ or $z=1$.

If $u=1$, then $z=-1$ or $z=1$. If $E=1$, then $z=-1$ or $z=u$. If $z=1$ and $u \neq 1$, then $E=-1$ which is absurd.

\subsection{\it The case $\zeta \neq -1$, $bz=c$ } We have $$bz=z+uEz-Ez+uz^2-z^2=u^2z-uz+z+u^2E-uE=c$$
which yields
$$z(u-1)(E+z)=u(u-1)(E+z).$$

In the next two subsections, we will see what happens when $u=1$ and $E=1$. Thus, for the moment we may assume that $u\neq 1$ and $E\neq 1$.

If $z=-E$, then $b=1$, and combining Equations~\ref{(1)} and \ref{(2)} yields $\zeta= \pm E$ and $q=1$.
If $z \neq -E$, then $z=u$, and combining  Equations~\ref{(1)} and \ref{(3)} yields a contradicition.

\subsection{\it The case  $\zeta \neq -1$, $u=1$}

If $\zeta \neq -1$ and $u=1$, Equation~\ref{(1)} becomes
\begin{equation}\label{group1}
q=\frac{\zeta^2+\zeta}{\zeta+z^2}.\end{equation}
Moreover, we have $b=1$, so
Equation~\ref{(2)} becomes
\begin{equation}\label{group2}
q=\frac{\zeta^2+z^2\zeta}{z^2\zeta+z^2}.\end{equation}
Combining Equations~\ref{group1} and \ref{group2} gives us
$$\zeta^2(z^2-1)=z^2(z^2-1),$$
which holds only if $\zeta=z$ or $\zeta=-z$ or $z=1$ or $z=-1$.

If $\zeta= \pm z$, then $q=1$. If  $z=\pm 1$, then $q=\zeta$.

\subsection{\it The case  $\zeta \neq -1$, $E=1$}

Suppose that  $z^2 \neq b$ (the case $z^2=b$ has been completely covered in \S\ref{z^2=b}). Then Equation~\ref{(5)} becomes:
\begin{equation}\label{quad1}
\zeta^2=\frac{zu-z}{u} \zeta + \frac{z^2}{u}.
\end{equation}
The above quadratic equation has two solutions: $\zeta=z$ or $\zeta=-z/u$.
If $\zeta=z$, then Equation~\ref{(1)} yields $q=u$. If $\zeta=-z/u$, then Equation~\ref{(1)} yields $q=1/u$.

\subsection{\it The case  $\zeta \neq -1$, $z=(1-E+u+uE)/2$}

As in the previous subsection, we may assume that  $z^2 \neq b$. Then Equation~\ref{(5)} becomes:
\begin{equation}\label{quad2}
\zeta^2=\frac{-1+E-2u-Eu^2-u^2}{2u} \zeta + \frac{-1+2E-2u-2Eu^2-E^2-u^2+2E^2u-E^2u^2}{4u}.
\end{equation}
The above quadratic equation has two solutions: $\zeta=-z$ or $\zeta=-z/u$.
If $\zeta=-z$, then Equation~\ref{(1)} yields $q=-u$. If $\zeta=-z/u$, then Equation~\ref{(1)} yields $q=-1/u$.

\subsection{\it Summarizing}\label{table}

The cases below are the only cases where $D_{\mathcal{H}}=D_{\rm Y}$ and
$$
\frac{\tau(G_i^m)}{{\rm tr}_{_S}(g_i^m)}=\left(\frac{\zeta}{z}\right)^{m} \,\,\,\,\text{for } m \in \{2,3\}.
$$
One can easily check, using (\ref{hpowersev}), (\ref{hpowersod}), (\ref{powersev}) and (\ref{powerseod}), that the above equality holds for all $m \in \mathbb{N}$.

$ $
\begin{center}
\begin{tabular}{|c|c|c|c|c|c|}
\hline
{\bf Case}&$q$ & $\zeta$ & $u$ & $z$ & $E$\\
\hline
{\bf 1}&$1$ & $z$ & $1$ & $\mathbb{C}^*$ & any\\
{\bf 2}&$1$ & $-z$ & $1$ & $\mathbb{C}^*$ & any\\
{\bf 3}&$\mathbb{C}^*$ & $q$ & $1$ &$1$ & any\\
{\bf 4}&$\mathbb{C}^*$ & $q$ & $1$ & $-1$ & any\\
{\bf 5}&$\mathbb{C}^*$ & $-1$ & $1$ &$1$ & any\\
{\bf 6}&$\mathbb{C}^*$ & $-1$ & $1$ & $-1$ & any\\
{\bf 7}&$1$ & $E$ & $\mathbb{C}^*$ & $-E$ & any\\
{\bf 8}&$1$ & $-E$ & $\mathbb{C}^*$ & $-E$ & any\\
{\bf 9}&$\mathbb{C}^*$ & $q$ & $\mathbb{C}^*$ & $-1$  & $1$\\
{\bf 10}&$\mathbb{C}^*$  & $q$ & $\mathbb{C}^*$ &$u$ & $1$\\
{\bf 11}&$\mathbb{C}^*$ & $-1$ & $\mathbb{C}^*$ & $-1$  & $1$\\
{\bf 12}&$\mathbb{C}^*$  & $-1$ & $\mathbb{C}^*$ &$u$ & $1$\\
{\bf 13}&$u$ & $z$ & $\mathbb{C}^*$ & $\mathbb{C}^*$  & $1$\\
{\bf 14}&$1/u$ & $-z/u$ & $\mathbb{C}^*$ &$\mathbb{C}^*$ & $1$\\
{\bf 15}&$-u$ & $-z$ & $\mathbb{C}^*$ & $(1-E+u+uE)/2$  & any\\
{\bf 16}&$-1/u$ & $-z/u$ & $\mathbb{C}^*$ &$(1-E+u+uE)/2 $ & any\\
\hline
\end{tabular}
\end{center}\
$ $

\subsection{\it Dismissing two cases}

Let us now take $\alpha=\sigma_1\sigma_2^2\sigma_1\sigma_2^2$ in $B_n$ for $n\geq 3$.
We have
$$
\tau(\pi(\alpha))=\tau(G_1G_2^2G_1G_2^2)= (q^2\zeta-2q\zeta+\zeta)(q^2\zeta-q\zeta+\zeta+q^2-q)+(2q^2\zeta-2q\zeta+q^2)(q\zeta-\zeta+q).
$$
and
$$
\begin{array}{rcccl}
{\rm tr}_{_S}(\delta(\alpha))&=&{\rm tr}_{_S}(g_1g_2^2g_1g_2^2)&=&{\rm tr}_{_S}(g_1^2)+2(u-1)\,{\rm tr}_{_S}(g_1^2e_2)+2(u-1)\,{\rm tr}_{_S}(g_1^2g_2e_2)\\
&&&&(u-1)^2\left[{\rm tr}_{_S}(g_1e_2g_1e_2)+2{\rm tr}_{_S}(g_1e_2g_1g_2e_2)+{\rm tr}_{_S}(g_1g_2e_2g_1g_2e_2)\right].
\end{array}
$$
Here, the fact that  $X_S = \{ {\rm x}_1,\ldots, {\rm x}_{d-1}\}$ is a solution of the E--system simplifies calculations a lot.
For example, we automatically deduce that ${\rm tr}_{_S}(g_1^2e_2)\stackrel{\rm E}{=} E\,{\rm tr}_{_S}(g_1^2)$. Now let us see what happens when we try to calculate
${\rm tr}_{_S}(g_1e_2g_1e_2)$. This is always equal to:
$$\frac{1}{d^2}\sum_{k=0}^{d-1}\sum_{m=0}^{d-1}{\rm tr}_{_S}(g_1t_2^kt_3^{-k}g_1t_2^mt_3^{-m})
=\frac{1}{d^2}\sum_{k=0}^{d-1}\sum_{m=0}^{d-1}{\rm tr}_{_S}(g_1^2t_1^kt_2^mt_3^{-k-m})=
\frac{1}{d^2}\sum_{k=0}^{d-1}\sum_{m=0}^{d-1}{\rm x}_{d-k-m}{\rm tr}_{_S}(g_1^2t_1^kt_2^m)
$$
We have
$$\begin{array}{rll}
{\rm tr}_{_S}(g_1^2t_1^kt_2^m)&=&{\rm tr}_{_S}(t_1^kt_2^m)+(u-1)\,{\rm tr}_{_S}(e_1t_1^kt_2^m)+(u-1)\,{\rm tr}_{_S}(g_1e_1t_1^kt_2^m)\\ \\
&=&{\rm x}_k{\rm x}_m+(u-1)\frac{1}{d}\sum_{l=0}^{d-1}{\rm x}_{k+l}{\rm x}_{m-l}+(u-1)z{\rm x}_{k+m}\\ \\
&\stackrel{\rm E}{=}& {\rm x}_k{\rm x}_m+(u-1)E\,{\rm x}_{k+m}+(u-1)z{\rm x}_{k+m}
\end{array}$$
Thus,
$$\begin{array}{rll}
{\rm tr}_{_S}(g_1e_2g_1e_2)&=&\frac{1}{d^2}\sum_{k=0}^{d-1}\sum_{m=0}^{d-1}{\rm x}_{d-k-m}{\rm x}_k{\rm x}_m+[(u-1)(E+z)]
\frac{1}{d^2}\sum_{k=0}^{d-1}\sum_{m=0}^{d-1}{\rm x}_{d-k-m}{\rm x}_{k+m}\\ \\
&\stackrel{\rm E}{=} &\frac{1}{d}\sum_{k=0}^{d-1}{\rm x}_{d-k}{\rm x}_kE+ [(u-1)(E+z)]\frac{1}{d}\sum_{k=0}^{d-1}E\\ \\
&\stackrel{\rm E}{=} &E^2+[(u-1)(E+z)]E=E\,(uE+uz-z)=E\,{\rm tr}_{_S}(e_1g_1^2)
\end{array}
$$
We finally obtain that
$$
{\rm tr}_{_S}(g_1g_2^2g_1g_2^2) \stackrel{\rm E}{=} b(2b-1)
+(u-1)^2(E+uz+z)(uE+uz-z)+u(u-1)^2z^2.
$$
In Cases 1--14, we obtain:
$$
\frac{\tau(\pi(\alpha))}{{\rm tr}_{_S}(\delta(\alpha))}=\frac{\tau(G_1G_2^2G_1G_2^2)}{{\rm tr}_{_S}(g_1g_2^2g_1g_2^2)}\stackrel{\rm E}{=}\left(\frac{\zeta}{z}\right)^{6}=\left(\frac{\zeta}{z}\right)^{\epsilon(\alpha)}.
$$
\\
Cases 15 and 16 collapse, unless $E=1$ or $u= \pm 1$. However:
\begin{itemize}
\item $\text{Case }15,\,E=1  \text{ is covered by Case }10$\,;
\item $\text{Case }15,\,u=1  \text{ is covered by Case }3$\,;
\item $\text{Case }15,\,u=-1  \text{ is covered by Case }7$\,;
\item $\text{Case }16,\,E=1  \text{ is covered by Case }12$\,;
\item $\text{Case }16,\,u=1  \text{ is covered by Case }3$\,;
\item $\text{Case }16,\,u=-1  \text{ is covered by Case }8$\,.
\end{itemize}

\section{The only cases where $P$ and $\Delta_S$ coincide} \label{sameinvs}

We will show that the cases where the invariants $P$ and $\Delta_S$ coincide are precisely the Cases~1--14 in the table of \S\ref{table}.
We have already shown that if $P$ and $\Delta_S$ coincide, then we must be in one of these cases. We will now show that
in all these cases, $P$ and $\Delta_S$ do coincide. Note that these results hold for any $d\in \mathbb{N}$ and for any non-empty subset $S$ of  ${\mathbb Z}/d{\mathbb Z}$.

\subsection{\it General methodology}\label{first section}

We already know (see \S\ref{table}) that, for $\a = \s_i^m \in B_n$ $(1 \leq i \leq n-1)$, in Cases~1--14 we have
\begin{equation}\label{n=1}
D_{\mathcal{H}}=D_{\rm Y} \,\,\,\,\text{and}\,\,\,\,\frac{\tau(\pi(\s_i^m))}{{\rm tr}_{_S}(\delta(\s_i^m))}=\left(\frac{\zeta}{z}\right)^{m} \,\,\,\,\text{for all } m \in \mathbb{N}.
\end{equation}
In particular, following (\ref{hpowersev}), (\ref{hpowersod}), (\ref{powersev}) and (\ref{powerseod}), in Cases~3--6 and 9--12 we have
\begin{equation}\label{n=2,3-6}
\tau(\pi(\s_i^m))=\zeta^m  \,\,\,\,\text{and}\,\,\,\,{\rm tr}_{_S}(\delta(\s_i^m))=z^m \,\,\,\,\text{for all } m \in \mathbb{N}.
\end{equation}

We need to show that
\begin{equation}\label{finalresult}
\frac{\tau(\pi(\alpha))}{{\rm tr}_{_S}(\delta(\alpha))}=\left(\frac{\zeta}{z}\right)^{\epsilon(\alpha)} \,\,\,\,\text{for all } \a \in B_n,
\end{equation}
where $\epsilon(\alpha)$ is the sum of the exponents of the braid generators in the braid word $\alpha$.
In Cases~3--6 and 9--12 we will even show that
\begin{equation}\label{3-6 tau}
{\tau(\pi(\alpha))}=\zeta^{\epsilon(\alpha)}\,\,\,\,\text{for all } \a \in B_n
\end{equation}
and
\begin{equation}\label{3-6 t}
{{\rm tr}_{_S}(\delta(\alpha))}=z^{\epsilon(\alpha)}\,\,\,\,\text{for all } \a \in B_n.
\end{equation}
We will prove  the above by induction on the following non-negative number:
$$\nu(\a):= |\,\text{sum of all negative exponents of the braid generators in } \a\,|.$$
Note that $\nu(\a)$ depends on the expression of $\a$ in terms of braid generators.

For the inductive step, we will show  (\ref{finalresult}) in  Cases ~7--8 and 13--14 (respectively (\ref{3-6 tau}) and (\ref{3-6 t}) in Cases~3--6 and 9--12) for $\alpha \s_i^{-1}$ $(\a \in B_n,\,1\leq i \leq n-1)$ using the  induction hypothesis. The formulas for $\pi( \s_i^{-1})= G_i^{-1}$ and  $\delta(\s_i^{-1})= g_i^{-1}$ are given respectively by Equations \ref{hinv} and \ref{invrs}.

For the first step ($\nu(\a)=0$), we will proceed by induction on $n$.
Set
\begin{center}
$B_n^{+}:=\{\, \a \in B_n \,\,|\,\, \nu(\a)=0\,\}$.\footnote{ This set is known as the \emph{braid monoid}, see, for example, \cite[Chapter 4]{gp}.}
\end{center}
The cases $n=1$ and $n=2$ are taken care of by (\ref{n=1}) for Cases ~7--8 and 13--14  (respectively (\ref{n=2,3-6}) for Cases~3--6 and 9--12). We will only need to prove
(\ref{finalresult}) (respectively (\ref{3-6 tau}) and (\ref{3-6 t}))
 for $B_{n+1}^{+}$ assuming that it holds for $B_n^{+}$. To do this we will use a
 second induction on the non-negative number $\epsilon(\beta)+\epsilon_n(\beta)$, where
 $\epsilon_n(\beta)$ denotes the sum of the exponents of the braid generator $\s_n$ in the braid word $\beta \in B_{n+1}^{+}$.
 Note that $\epsilon(\beta)$ is uniquely defined for $\beta$, whereas $\epsilon_n(\beta)$ depends on the expression of $\beta$ in terms of braid generators.
 However, we always have
 $\epsilon(\beta) \geq \epsilon_n(\beta)$.

  If $\epsilon(\beta)+\epsilon_n(\beta)=0$, then $\beta=1$ and there is nothing to prove.
 If $\epsilon(\beta)+\epsilon_n(\beta)=1$, then $\epsilon(\beta)=1$ and $\epsilon_n(\beta)=0$.
 Hence, $\beta = \s_i$ for some $1 \leq i \leq n-1$ and the desired result holds. Now assume that $\epsilon(\beta)+\epsilon_n(\beta)>1$ and that the result holds for smaller values of $\epsilon+\epsilon_n$. We will distinguish three cases:
 \begin{itemize}
 \item If $\epsilon_n(\beta)=0$, then $\beta \in B_n^+$, and the induction hypothesis on $n$ yields the desired result.
 \item If  $\epsilon_n(\beta)=1$, then there exist $\a_1,\,\a_2 \in B_n^+$ such that $\beta=\a_1\s_n\a_2$. We have
 $$\tau(\pi(\beta))=\zeta \, \tau(\pi(\a_1\a_2))\,\,\,\,\text{and}\,\,\,\,{\rm tr}_{_S}(\delta(\beta))=z \, {\rm tr}_{_S}(\delta(\a_1\a_2)).$$
 The induction hypothesis on $n$ yields the rest.
 \item If $\epsilon_n(\beta) \geq 2$, then there exist $\a \in B_n^+$ and $\beta_1,\,\beta_2 \in B_{n+1}^+$ such that
 $\beta=\beta_1\s_n \a \s_n \beta_2$. We will  need the following lemma:
 \end{itemize}

\begin{lem}\label{psalidi}
Let $\a \in B_n^{+}$. Then one of the following hold:
\begin{enumerate}[(i)]
\item $\s_n\a\s_n=\a_1 \s_n \a_2$, for some $\a_1,\a_2 \in B_n^{+}$, or \smallbreak
\item $\s_n\a\s_n=\beta_1 \s_j^2 \beta_2$, for some $\beta_1,\beta_2 \in B_{n+1}^{+}$ and $1 \leq j \leq n$.
\end{enumerate}
\end{lem}

\begin{proof} We will proceed by induction on $n$. If $n=1$, then $\a=1$ and we are in  Case {\em (ii)}. Assume that the above holds for $1,2,\ldots,n-1$. We will show that it holds for $n$. We will proceed by induction on the number $\epsilon_{n-1}(\alpha)$, that is, the sum of the exponents of the braid generator $\s_{n-1}$ in the braid word $\a$:
\begin{itemize}
\item If  $\epsilon_{n-1}(\alpha)=0$, then $\s_n$ commutes with $\a$, and $$\s_n\a\s_n=\a \s_n^2.$$
\item If  $\epsilon_{n-1}(\alpha)=1$, then there exist $b_1,\,b_2 \in B_{n-1}^+$ such that $\a=b_1\s_{n-1}b_2$. We have:
$$ \s_n\a\s_n=b_1\s_n\s_{n-1}\s_nb_2=b_1\s_{n-1}\s_n\s_{n-1}b_2=\a_1 \s_n \a_2,$$
where $\a_1=b_1\s_{n-1} \in B_n^{+}$ and $\a_2=\s_{n-1}b_2 \in B_n^{+}$.
\item If  $\epsilon_{n-1}(\alpha) \geq 2$, then there exist $b \in B_{n-1}^+$ and $\a_1,\a_2 \in B_n^{+}$ such that
$\a=\a_1\s_{n-1}b\s_{n-1}\a_2$. Then, by the induction hypothesis on $n$, one of the following hold:\smallbreak
\begin{enumerate}[(i)]
\item $\s_{n-1}b\s_{n-1}=b_1 \s_{n-1} b_2$, for some $b_1,b_2 \in B_{n-1}^{+}$, or \smallbreak
\item $\s_{n-1}b\s_{n-1}=\alpha_1' \s_j^2 \a_2'$, for some $\a_1',\a_2' \in B_{n}^{+}$ and $1 \leq j \leq n-1$.
\end{enumerate}
In Case (i), the induction hypothesis on $\epsilon_{n-1}(\alpha)$ yields the desired result. In Case (ii), we obtain:
$$ \s_n\a\s_n= \s_n\a_1\a_1' \s_j^2 \a_2'\a_2\s_n=\beta_1 \s_j^2 \beta_2,$$
where $\beta_1=\s_n\a_1\a_1' \in B_{n+1}^{+}$ and $\beta_2=\a_2'\a_2\s_n \in B_{n+1}^{+}$.
\end{itemize}
\end{proof}

Applying now the above lemma to the word $\beta=\beta_1\s_n \a \s_n \beta_2$, where $\a \in B_n^+$ and $\beta_1,\,\beta_2 \in B_{n+1}^+$,
 we obtain that one of the following hold:
\begin{enumerate}[(i)]
\item $\beta=\beta_1\a_1 \s_n \a_2\beta_2$, for some $\a_1,\a_2 \in B_n^{+}$, or \smallbreak
\item $\beta=\beta_1\beta_1' \s_j^2 \beta_2'\beta_2$, for some $\beta_1',\beta_2' \in B_{n+1}^{+}$ and $1 \leq j \leq n$.
\end{enumerate}
The induction hypothesis on $\epsilon(\beta)+\epsilon_n(\beta)$ covers Case (i), since
$\epsilon_n(\beta_1\a_1 \s_n \a_2\beta_2)<\epsilon_n(\beta_1\s_n \a \s_n \beta_2 )$.
Therefore, it is enough to prove (\ref{finalresult}) (respectively (\ref{3-6 tau}) and (\ref{3-6 t})) in Case (ii).
Since $\tau$ and ${\rm tr}_{_S}$ are trace functions, we deduce that we will check the desired equalities on all words of the form:
$$ \beta \s_j^2 \quad{(\beta  \in B_{n+1}^{+}, 1\leq j \leq n)}.$$

To summarize: In order to prove Equality \ref{finalresult} in Cases~7--8 and 13--14, and Equalities \ref{3-6 tau} and \ref{3-6 t} in Cases 3--6 and 9--12, we will show that these equalities hold on all words of the form
$$ \beta \s_j^2 \quad{(\beta  \in B_{n+1}^{+}, 1\leq j \leq n)},$$
assuming the induction hypotheses on $n$ and on $\epsilon+\epsilon_n$, and all words of the form
$$\alpha \s_i^{-1} \quad{(\a \in B_n,\,1\leq i \leq n-1)},$$
assuming the induction hypothesis on $\nu$.

\subsection{\it The Cases 1 and 2}

In the first two cases, although our general methodology applies,  we prefer to use the following, simpler approach. Since $q=1$,  the quadratic relation (\ref{hecke})($\mathrm{h}$) in the Iwahori-Hecke algebra $\mathcal{H}_n(1) \cong \mathfrak{S}_n$ becomes $G_i^2=1$.
Similarly, since $u=1$, the   quadratic relation (\ref{quadr}) in the Yokonuma-Hecke algebra ${\rm Y}_{d,n}(1)$ becomes  $g_i^2=1$.
Therefore, there exist two natural isomorphisms $\iota^{+}$ and $\iota^{-}$ between $\pi(\mathbb{C}B_n) \cong \mathcal{H}_n(1)$ and $\delta(\mathbb{C}B_n)$ given by
$\iota^{+}(G_i)=g_i$ and $\iota^{-}(G_i)=-g_i$ respectively.
Now, if we take $\zeta=z$, then ${\rm tr}_{_S} \circ \iota^{+}$ is a Markov trace on $\mathcal{H}_n(1)$ that satisfies all three rules of Theorem~\ref{otrace}.
The uniqueness of the Ocneanu trace yields ${\rm tr}_{_S} \circ \iota^{+}=\tau$. So in Case 1 we have:
$$
\frac{\tau\left(\pi(\alpha)\right)}{{\rm tr}_{_S}(\delta(\alpha))}=
\frac{{\rm tr}_{_S}\left( \iota^{+} \left(\pi(\alpha)\right) \right)}{{\rm tr}_{_S}\left( \delta(\alpha) \right)}=
\frac{{\rm tr}_{_S} \left(\delta(\alpha)\right)}{{\rm tr}_{_S} \left(\delta(\alpha)\right)} = 1 =
\left(\frac{z}{z}\right)^{\epsilon(\alpha)}=
\left(\frac{\zeta}{z}\right)^{\epsilon(\alpha)} \,\,\,\,\text{for all } \a \in B_n.
$$
Similarly,
if we take $\zeta=-z$, then ${\rm tr}_{_S} \circ \iota^{-}$ is a Markov trace that satisfies all three rules of Theorem~\ref{otrace}.
Therefore, we obtain ${\rm tr}_{_S} \circ \iota^{-}=\tau$.
So in Case 2 we have:
$$
\frac{\tau\left(\pi(\alpha)\right)}{{\rm tr}_{_S}\left(\delta(\alpha)\right)}=
\frac{{\rm tr}_{_S}\left(\iota^{-}\left(\pi(\alpha)\right)\right)}{{\rm tr}_{_S}\left(\delta(\alpha)\right)}=
\frac{(-1)^{\epsilon(\a)}{\rm tr}_{_S}\left(\delta(\alpha)\right)}{{\rm tr}_{_S}\left(\delta(\alpha)\right)}=(-1)^{\epsilon(\a)}=
\left(\frac{-z}{z}\right)^{\epsilon(\alpha)}=
\left(\frac{\zeta}{z}\right)^{\epsilon(\alpha)} \,\,\,\,\text{for all } \a \in B_n.
$$

\subsection{\it The Cases 3--6}\label{3-6}
Following our general methodology, let first $\beta \in B_{n+1}^{+}$ and $1 \leq j \leq n$.
 If $\zeta=q$, we have
$$
{\tau\left(\pi(\beta \s_j^2)\right)}=(q-1)\,{\tau\left(\pi(\beta \s_j)\right)}+q\,{\tau\left(\pi(\beta)\right)}\stackrel{\rm ind. \,hyp.}{=}(q-1)\, q^{\epsilon(\beta)+1}+q\cdot q^{\epsilon(\beta)}=q^{\epsilon(\beta)+2}=
\zeta^{\epsilon(\beta\s_j^2)}.
$$
If $\zeta=-1$, we have
$$
{\tau\left(\pi(\beta \s_j^2)\right)}=
(q-1)\,{\tau\left(\pi(\beta \s_j)\right)}+q\, {\tau\left(\pi(\beta)\right)}\stackrel{\rm ind. \,hyp.}{=}
(q-1)\,(-1)^{\epsilon(\beta)+1}+q \, (-1)^{\epsilon(\beta)}=(-1)^{\epsilon(\beta)+2}=
\zeta^{\epsilon(\beta\s_j^2)}.
$$
If $u=1$ and $z=1$, we have
$$
{{\rm tr}_{_S}\left(\delta(\beta \s_j^2)\right)}=
{{\rm tr}_{_S}\left(\delta(\beta)\right)} \stackrel{\rm ind. \,hyp.}{=} 1= z^{\epsilon(\beta\s_j^2)}.
$$
If $u=1$ and $z=-1$, we have
$$
{{\rm tr}_{_S}\left(\delta(\beta \s_j^2)\right)} = {{\rm tr}_{_S}\left(\delta(\beta)\right)}\stackrel{\rm ind. \,hyp.}{=} (-1)^{\epsilon(\beta)}=(-1)^{\epsilon(\beta)+2}=
z^{\epsilon(\beta\s_j^2)}.
$$

Now let $\a \in B_n$ and $1 \leq i \leq n-1$. If $\zeta=q$, we have
$$
\tau\left(\pi(\a \s_i^{-1})\right)= q^{-1}\,{\tau\left(\pi(\a \s_i)\right)} +(q^{-1}-1)\,{\tau\left(\pi(\a)\right)}\stackrel{\rm ind. \,hyp.}{=}
q^{-1}q^{\epsilon(\a)+1}+(q^{-1}-1)\, q^{\epsilon(\a)}=
q^{\epsilon(\a)-1}=
\zeta^{\epsilon(\a \s_i^{-1})}.
$$
If $\zeta=-1$, we have
$$
{\tau\left(\pi(\a \s_i^{-1})\right)}=q^{-1}\,{\tau\left(\pi(\a \s_i)\right)}+(q^{-1}-1)\,{\tau\left(\pi(\a)\right)}\stackrel{\rm ind. \,hyp.}{=}q^{-1}(-1)^{\epsilon(\a)+1}+(q^{-1}-1)(-1)^{\epsilon(\a)}=(-1)^{\epsilon(\a)-1}=
\zeta^{\epsilon(\a \s_i^{-1})}.
$$
If $u=1$ and $z=1$, we have
$$
{{\rm tr}_{_S}\left(\delta(\a \s_i^{-1})\right)}={{\rm tr}_{_S}\left(\delta(\a \s_i)\right)}\stackrel{\rm ind. \,hyp.}{=}1=z^{\epsilon(\a \s_i^{-1})}.
$$
If $u=1$ and $z=-1$, we have
$$
{{\rm tr}_{_S}\left(\delta(\a \s_i^{-1})\right)}={{\rm tr}_{_S}\left(\delta(\a \s_i)\right)}\stackrel{\rm ind. \,hyp.}{=}(-1)^{\epsilon(\a)+1}=(-1)^{\epsilon(\a)-1}=z^{\epsilon(\a \s_i^{-1})}.
$$

Thus, we conclude that (\ref{3-6 tau}) and (\ref{3-6 t}) hold, whence we deduce (\ref{finalresult}).

\subsection{\it The Cases 7 and 8}

In order to show (\ref{finalresult}) in Cases~7 and~8, we will first show that
\begin{equation}\label{sqh}
\tau(h G_j^2)=\tau(h) \quad{\left(h \in \mathcal{H}_n(1),\,\,1 \leq j \leq n-1\right)}
\end{equation}
and
\begin{equation}\label{sqy0}
{\rm tr}_{_S}(y g_j^2)={\rm tr}_{_S}(y) \quad{\left(y \in {\rm Y}_{d,n}(u),\,\,1 \leq j \leq n-1\right)}.
\end{equation}
Note that (\ref{sqy0}) is equivalent to
\begin{equation}\label{sqy}
{\rm tr}_{_S}(y e_j)=-{\rm tr}_{_S}(y g_j e_j) \quad{\left(y \in {\rm Y}_{d,n}(u),\,\,1 \leq j \leq n-1\right)}.
\end{equation}
Equation \ref{sqh} is straightforward, since $G_j^2=1$, for all $j=1,\ldots,n-1$, in $\mathcal{H}_n(1)$. To prove
(\ref{sqy}) we will proceed by induction on $n$.
Recall that $z=-E$. It is enough to show that (\ref{sqy}) holds on the elements of the inductive basis of
${\rm Y}_{d,n}(u)$.

Let $n=2$. Let $y=t_1^kg_1t_1^l$ for some $k,\,l\in {\mathbb Z}/d{\mathbb Z}$. Then, by (\ref{edikrels}) and (\ref{traceaengn}), we have
$${\rm tr}_{_S}(y e_1)={\rm tr}_{_S}(t_1^kg_1t_1^le_1)={\rm tr}_{_S}(t_1^{k+l}g_1e_1)=
-E\,{\rm tr}_{_S}(t_1^{k+l})=-E\, {\rm x}_{k+l},$$
and, by (\ref{modular})($\mathrm{f}_2$), Lemma \ref{newequal} and (\ref{etraceanen}), we have
$${\rm tr}_{_S}(y g_1e_1)={\rm tr}_{_S}(t_1^kg_1t_1^lg_1e_1)=
 {\rm tr}_{_S}(t_1^kg_1^2e_1t_1^l)=
{\rm tr}_{_S}(t_1^{k+l}e_1)+(u-1)\,{\rm tr}_{_S}(t_1^{k+l}e_1)+(u-1)\,{\rm tr}_{_S}(t_1^{k+l}g_1e_1) \stackrel{\rm E}{=}  E\,{\rm x}_{k+l}.$$
So  (\ref{sqy}) holds. Now let  $y=t_1^{k}t_2^{l}$ for some $k,\,l\in {\mathbb Z}/d{\mathbb Z}$. Then, by (\ref{modular})($\mathrm{f}_2$), (\ref{edikrels}) and Lemma \ref{newequal}, we have
$${\rm tr}_{_S}(ye_1)={\rm tr}_{_S}(t_1^{k}t_2^{l}e_1)=
{\rm tr}_{_S}(t_1^{k+l}e_1)=E\,{\rm x}_{k+l}$$
and
$${\rm tr}_{_S}(yg_1e_1)={\rm tr}_{_S}(t_1^{k}t_2^{l}g_1e_1)={\rm tr}_{_S}(t_1^{k}g_1e_1t_1^{l})=-E\,{\rm x}_{k+l}.$$
So (\ref{sqy}) holds.

Now assume that (\ref{sqy}) holds for $n$. We will prove it for $n+1$.  Let $y=w_ng_ng_{n-1} \ldots g_i t_i^k$ for some $k\in {\mathbb Z}/d{\mathbb Z}$ and some $w_n \in {\rm Y}_{d,n}(u)$.
If $j<n$, then, following the definition of the trace and the induction hypothesis, we have
$${\rm tr}_{_S}(y e_j)={\rm tr}_{_S}(w_ng_ng_{n-1} \ldots g_i t_i^ke_j)
=-E\, {\rm tr}_{_S}(w_ng_{n-1} \ldots g_i t_i^k e_j)= E\, {\rm tr}_{_S}(w_ng_{n-1} \ldots g_i t_i^k g_je_j)=-{\rm tr}_{_S}(y g_je_j).$$
If $j=n$, then, we have to distinguish two cases: If $i=n$, then  we have,  by (\ref{edikrels}),
$${\rm tr}_{_S}(y e_n)={\rm tr}_{_S}(w_ng_nt_n^ke_n)=
{\rm tr}_{_S}(t_n^kw_ng_ne_n) = -E\,{\rm tr}_{_S}(t_n^kw_n),$$
and,  by (\ref{modular})($\mathrm{f}_2$), (\ref{edikrels}) and Lemma \ref{newequal},
$${\rm tr}_{_S}(y g_ne_n)={\rm tr}_{_S}(w_ng_nt_n^kg_ne_n)=
{\rm tr}_{_S}(t_n^kw_n g_n^2e_n)=u\, {\rm tr}_{_S}(t_n^kw_n e_n)+(u-1)\,{\rm tr}_{_S}(t_n^kw_n g_ne_n)=E\,{\rm tr}_{_S}(t_n^kw_n).$$
If $i<n$, then
 $${\rm tr}_{_S}(y e_n)= {\rm tr}_{_S}(w_ng_ng_{n-1} \ldots g_i t_i^ke_n) \stackrel{(\ref{edikrels})}{=}{\rm tr}_{_S}(w_ne_{n-1}g_ng_{n-1} \ldots g_i t_i^k)=-E\, {\rm tr}_{_S}( g_{n-1} \ldots g_i t_i^kw_ne_{n-1}).$$
Following the induction hypothesis, the latter is equal to
 $$E\, {\rm tr}_{_S}( g_{n-1} \ldots g_i t_i^kw_n g_{n-1}e_{n-1})= - {\rm tr}_{_S}(w_n g_{n-1}e_{n-1} g_n g_{n-1} \ldots g_i t_i^k) \stackrel{(\ref{edikrels})}{=}$$
$$ - {\rm tr}_{_S}(w_n g_{n-1}g_n g_{n-1} \ldots g_i t_i^ke_n)=- {\rm tr}_{_S}(w_ng_n g_{n-1}g_ng_{n-2} \ldots g_i t_i^k e_n)=- {\rm tr}_{_S}(y g_ne_n).
$$

Finally, let $y=w_nt_{n+1}^k$ for some $k\in {\mathbb Z}/d{\mathbb Z}$ and some $w_n \in {\rm Y}_{d,n}(u)$.
If $j<n$, then, following the definition of the trace and the induction hypothesis, we have
$${\rm tr}_{_S}(y e_j)={\rm tr}_{_S}(w_nt_{n+1}^ke_j)=
{\rm x}_k{\rm tr}_{_S}(w_ne_j)=-{\rm x}_k{\rm tr}_{_S}(w_ng_je_j)=-{\rm tr}_{_S}(y g_je_j).$$
If $j=n$, then, with repeated use of Lemma \ref{newequal} we obtain:
 $${\rm tr}_{_S}(y e_n)={\rm tr}_{_S}(w_nt_{n+1}^ke_n)=
 {\rm tr}_{_S}(w_n t_n^k e_n)\stackrel{\rm E}{=}  E\,{\rm tr}_{_S}(w_n t_n^k)=-{\rm tr}_{_S}(w_n t_n^ke_ng_n)=-{\rm tr}_{_S}(w_n t_{n+1}^ke_ng_n)=
 -{\rm tr}_{_S}(yg_ne_n).$$

Equations \ref{sqh} and \ref{sqy0} imply the following for the inverses of the generators:
\begin{equation}\label{last1}
\tau(h G_j^{-1})=\tau(hG_j) \quad{\left(h \in \mathcal{H}_n(1),\,\,1 \leq j \leq n-1\right)}
\end{equation}
and
\begin{equation}\label{last2}
{\rm tr}_{_S}(y g_j^{-1})={\rm tr}_{_S}(yg_j) \quad{\left(y \in {\rm Y}_{d,n}(u),\,\,1 \leq j \leq n-1\right)}.
\end{equation}

We are now ready to prove (\ref{finalresult}). Let $\beta \in B_{n+1}^{+}$ and $1 \leq j \leq n$. Following Equations \ref{sqh} and \ref{sqy0}, we obtain
$$\frac{\tau(\pi(\beta\s_j^2))}{{\rm tr}_{_S}(\delta(\beta\s_j^2))}=\frac{\tau(\pi(\beta)G_j^2)}{{\rm tr}_{_S}(\delta(\beta)g_j^2)}=
\frac{\tau(\pi(\beta))}{{\rm tr}_{_S}(\delta(\beta))}\stackrel{\rm ind. \,hyp.}{=}\left(\frac{\zeta}{z}\right)^{\epsilon(\beta)}=\left(\frac{\zeta}{z}\right)^{\epsilon(\beta)+2}=\left(\frac{\zeta}{z}\right)^{\epsilon(\beta\s_j^2)},
$$
since
\begin{equation}\label{right ratio}
\frac{\zeta}{z}=-1 \text{ in Case 7 and } \frac{\zeta}{z}=1 \text{ in Case 8.}
\end{equation}

Now let $\a \in B_n$ and $1 \leq i \leq n-1$. Following Equations \ref{last1} and \ref{last2}, we obtain
$$\frac{\tau(\pi(\a\s_i^{-1}))}{{\rm tr}_{_S}(\delta(\a\s_i^{-1}))}=\frac{\tau(\pi(\a)G_i^{-1})}{{\rm tr}_{_S}(\delta(\a)g_i^{-1})}=
\frac{\tau(\pi(a)G_i)}{{\rm tr}_{_S}(\delta(a)g_i)}\stackrel{\rm ind. \,hyp.}{=}\left(\frac{\zeta}{z}\right)^{\epsilon(\a)+1}=\left(\frac{\zeta}{z}\right)^{\epsilon(\a)-1}=\left(\frac{\zeta}{z}\right)^{\epsilon(\a\s_i^{-1})},
$$
again because of (\ref{right ratio}).
Therefore, in both Cases 7 and 8, our general methodology yields (\ref{finalresult}).

\subsection{\it The Cases 9--12}

We will show that, in these cases, (\ref{3-6 tau}) and (\ref{3-6 t}) hold again for all $\alpha \in B_n$. On the Hecke algebra side, we have already shown in \S\ref{3-6}
that (\ref{3-6 tau}) holds when $\zeta=q$ or $\zeta=-1$. Thus, it remains to show (\ref{3-6 t}) for Cases  9--12.

Let $\beta \in B_{n+1}^{+}$ and $1 \leq j \leq n$.
We have
$$
{{\rm tr}_{_S}\left(\delta(\beta \s_j^2)\right)}={{\rm tr}_{_S}\left(\delta(\beta) g_j^2\right)}={{\rm tr}_{_S}\left(\delta(\beta)\right)}+(u-1)\,{\rm tr}_{_S}\left(\delta(\beta)e_j\right)+(u-1)\,{\rm tr}_{_S}\left(\delta(\beta)e_jg_j\right).
$$
By Proposition \ref{E=1}, if $E=1$, the latter is equal to
$$
{{\rm tr}_{_S}\left(\delta(\beta)\right)}+(u-1)\,{\rm tr}_{_S}\left(\delta(\beta)\right)+(u-1)\,{\rm tr}_{_S}\left(\delta(\beta)g_j\right).
$$
On the one hand, if $z=u$, we deduce, using the induction hypothesis, that
$$
{{\rm tr}_{_S}\left(\delta(\beta \s_j^2)\right)}=u^{\epsilon(\beta)}+(u-1)\, u^{\epsilon(\beta)}+(u-1)\, u^{\epsilon(\beta)+1}=u^{\epsilon(\beta)+2}=
z^{\epsilon(\beta\s_j^2)}.
$$
On the other hand, if $z=-1$, we deduce, using the induction hypothesis, that
$$
{{\rm tr}_{_S}\left(\delta(\beta \s_j^2)\right)}= (-1)^{\epsilon(\beta)}+(u-1)\,(-1)^{\epsilon(\beta)}+(u-1)\, (-1)^{\epsilon(\beta)+1}=(-1)^{\epsilon(\beta)+2}=
z^{\epsilon(\beta\s_j^2)}.
$$

Now let $\a \in B_n$ and $1 \leq i \leq n-1$. Using formula (\ref{invrs}), we obtain
$$
{{\rm tr}_{_S}\left(\delta(\a \s_i^{-1})\right)}={{\rm tr}_{_S}\left(\delta(\a) g_i^{-1}\right)}={{\rm tr}_{_S}\left(\delta(\a)g_i\right)}+(u^{-1} - 1)\, {{\rm tr}_{_S}\left(\delta(\a)e_i\right)}+ (u^{-1} - 1)\, {{\rm tr}_{_S}\left(\delta(\a)e_i g_i\right)}.
$$
By Proposition \ref{E=1}, if $E=1$, the latter is equal to
$$
{{\rm tr}_{_S}\left(\delta(\a)g_i\right)}+(u^{-1}-1){\rm tr}_{_S}\left(\delta(\a)\right)+(u^{-1}-1){\rm tr}_{_S}\left(\delta(\a)g_i\right).
$$
On the one hand, if $z=u$, we deduce, using the induction hypothesis, that
$$
{{\rm tr}_{_S}\left(\delta(\a \s_i^{-1})\right)}= u^{\epsilon(\a)+1}+(u^{-1}-1)\,u^{\epsilon(\a)}+(u^{-1}-1)\,u^{\epsilon(\a)+1}=u^{\epsilon(\a)-1}=z^{\epsilon(\a\s_i^{-1})}.
$$
On the other hand, if $z=-1$, we deduce, using the induction hypothesis, that
$$
{{\rm tr}_{_S}\left(\delta(\a \s_i^{-1})\right)}= (-1)^{\epsilon(\a)+1}+(u^{-1}-1)(-1)^{\epsilon(\a)}+(u^{-1}-1)(-1)^{\epsilon(\a)+1}=
(-1)^{\epsilon(\a)-1}=z^{\epsilon(\a\s_i^{-1})}.
$$

Thus, we conclude that (\ref{3-6 t}) holds. Since (\ref{3-6 tau}) also holds, we deduce that (\ref{finalresult}) holds.

\subsection{\it The Cases 13 and 14}\label{last section}

The Cases 13~and~14 have been covered by Corollaries \ref{firstcompinv} and \ref{secondcompinv} respectively. Nevertheless, we will see here how our general methodology applies also to these cases.

Let $\beta \in B_{n+1}^+$ and $1 \leq j \leq n$. We have
$$
\tau\left(\pi(\beta\s_j^2)\right)=\tau\left(\pi(\beta)G_j^2\right)= (q-1)\,\tau\left(\pi(\beta)G_j\right) +q\,\tau\left(\pi(\beta)\right)= (q-1)\,\tau\left(\pi(\beta\s_j)\right)+q\,\tau\left(\pi(\beta)\right)
$$
and
$$
{\rm tr}_{_S}\left(\delta(\beta\s_j^2)\right)=
{\rm tr}_{_S}\left(\delta(\beta)g_j^2\right)= {\rm tr}_{_S}\left(\delta(\beta)\right)+(u-1)\,{\rm tr}_{_S}\left(\delta(\beta)e_j\right)+(u-1)\,{\rm tr}_{_S}\left(\delta(\beta)e_jg_j\right).
$$
Since $E=1$, by Proposition \ref{E=1}, the last equation becomes:
 $$
 {\rm tr}_{_S}\left(\delta(\beta\s_j^2)\right)= {\rm tr}_{_S}\left(\delta(\beta)\right)+ (u-1)\,{\rm tr}_{_S}\left(\delta(\beta)\right)+(u-1)\,{\rm tr}_{_S}\left(\delta(\beta)g_j\right)=
(u-1)\,{\rm tr}_{_S}\left(\delta(\beta\s_j)\right)+u\,{\rm tr}_{_S}\left(\delta(\beta)\right).
$$
If $q=u$ and $\zeta=z$, the induction hypothesis on (\ref{finalresult})  yields:
$$
\tau\left(\pi(\beta\s_j^2)\right) = (u-1) \cdot 1^{\epsilon(\beta)+1}\cdot{\rm tr}_{_S}\left(\delta(\beta\s_j)\right) +u\cdot 1^{\epsilon(\beta)}\cdot{\rm tr}_{_S}(\delta(\beta))=
{\rm tr}_{_S}\left(\delta(\beta\s_j^2)\right),
$$
as desired. If $q=1/u$ and $\zeta=-z/u$, the induction hypothesis on (\ref{finalresult}) yields:
$$
\tau\left(\pi(\beta\s_j^2)\right)=\left(\frac{1}{u}-1\right) \left(\frac{-1}{u}\right)^{\epsilon(\beta)+1}{\rm tr}_{_S}(\delta(\beta\s_j))+\frac{1}{u}\left(\frac{-1}{u}\right)^{\epsilon(\beta)}{\rm tr}_{_S}(\delta(\beta))=\left(\frac{-1}{u}\right)^{\epsilon(\beta)+2}{\rm tr}_{_S}(\delta(\beta\s_j^2)),
$$
as desired.

Now let $\a \in B_n$ and $1 \leq i \leq n-1$. We have
$$
\tau\left(\pi(\alpha\s_i^{-1})\right)=q^{-1}\,\tau\left(\pi(\alpha\s_i)\right)+(q^{-1}-1)\,\tau\left(\pi(\alpha)\right)
$$
and
$$
{\rm tr}_{_S}\left(\delta(\a\s_i^{-1})\right) = {\rm tr}_{_S}\left(\delta(\a)g_i^{-1}\right)= {\rm tr}_{_S}\left(\delta(\a)g_i\right)+(u^{-1} - 1)\,{\rm tr}_{_S}\left(\delta(\a)e_i\right)+ (u^{-1} - 1)\,{\rm tr}_{_S}\left(\delta(\a)g_ie_i\right).
$$
Since $E=1$, by Proposition \ref{E=1}, the last equation becomes:
$$
{\rm tr}_{_S}\left(\delta(\a\s_i^{-1})\right)=
{\rm tr}_{_S}\left(\delta(\a)g_i\right)+(u^{-1} - 1)\,{\rm tr}_{_S}\left(\delta(\a)\right)+ (u^{-1} - 1)\,{\rm tr}_{_S}\left(\delta(\a)g_i\right)=
u^{-1}{\rm tr}_{_S}\left(\delta(\a\s_i)\right)+(u^{-1} - 1)\,{\rm tr}_{_S}\left(\delta(\a)\right).
$$
If $q=u$ and $\zeta=z$, the induction hypothesis on (\ref{finalresult}) yields:
$$
\tau\left(\pi(\alpha\s_i^{-1})\right) = u^{-1}\cdot 1^{\epsilon(\a)+1}\cdot{\rm tr}_{_S}(\delta(\a\s_i))+(u^{-1} - 1)\cdot 1^{\epsilon(\a)}\cdot{\rm tr}_{_S}(\delta(\a))=
{\rm tr}_{_S}\left(\delta(\a\s_i^{-1})\right),
$$
as desired. If $q=1/u$ and $\zeta=-z/u$, the induction hypothesis on (\ref{finalresult}) yields:
$$
\tau\left(\pi(\alpha\s_i^{-1})\right)=u  \left(\frac{-1}{u}\right)^{\epsilon(\a)+1}{\rm tr}_{_S}\left(\delta(\a\s_i)\right)+(u-1)\left(\frac{-1}{u}\right)^{\epsilon(\a)}{\rm tr}_{_S}(\delta(\a))
= \left(\frac{-1}{u}\right)^{\epsilon(\a)-1}{\rm tr}_{_S}(\delta(\a\s_i^{-1})),
$$
as desired.
Following our general methodology, (\ref{finalresult}) holds also in Cases 13 and 14.

\subsection{\it Conclusion}

The following result, proved in Subsections \ref{first section}--\ref{last section}, is the main result of this paper.

\begin{thm}\label{mainthm}
Let $X_S$ be a solution of the {\rm E}-system. Let ${\rm tr}_{_S}$ be the corresponding specialized Juyumaya trace on ${\rm Y}_{d,n}(u)$ with parameter $z$, and let
$\tau$ be the Ocneanu trace on $\mathcal{H}_{n}(q)$ with parameter $\zeta$. Let $E={\rm tr}_{_S}(e_i)$ for all $i=1,\ldots,n-1$.
Then $P=\Delta_S$ if and only if we are in one of the cases portrayed in the following table:

\begin{center}
\begin{tabular}{|c|c|c|c|c|c|}
\hline
{\bf Case}&$q$ & $\zeta$ & $u$ & $z$ & $E$\\
\hline
{\bf 1}&$1$ & $z$ & $1$ & $\mathbb{C}^*$ & any\\
{\bf 2}&$1$ & $-z$ & $1$ & $\mathbb{C}^*$ & any\\
{\bf 3}&$\mathbb{C}^*$ & $q$ & $1$ &$1$ & any\\
{\bf 4}&$\mathbb{C}^*$ & $q$ & $1$ & $-1$ & any\\
{\bf 5}&$\mathbb{C}^*$ & $-1$ & $1$ &$1$ & any\\
{\bf 6}&$\mathbb{C}^*$ & $-1$ & $1$ & $-1$ & any\\
{\bf 7}&$1$ & $E$ & $\mathbb{C}^*$ & $-E$ & any\\
{\bf 8}&$1$ & $-E$ & $\mathbb{C}^*$ & $-E$ & any\\
{\bf 9}&$\mathbb{C}^*$ & $q$ & $\mathbb{C}^*$ & $-1$  & $1$\\
{\bf 10}&$\mathbb{C}^*$  & $q$ & $\mathbb{C}^*$ &$u$ & $1$\\
{\bf 11}&$\mathbb{C}^*$ & $-1$ & $\mathbb{C}^*$ & $-1$  & $1$\\
{\bf 12}&$\mathbb{C}^*$  & $-1$ & $\mathbb{C}^*$ &$u$ & $1$\\
{\bf 13}&$u$ & $z$ & $\mathbb{C}^*$ & $\mathbb{C}^*$  & $1$\\
{\bf 14}&$1/u$ & $-z/u$ & $\mathbb{C}^*$ &$\mathbb{C}^*$ & $1$\\
\hline
\end{tabular}
\end{center}\
$ $
\end{thm}

\subsection{\it Comparing further $P$ and $\Delta_S$}

In Theorem \ref{mainthm} we give a necessary and sufficient condition for the invariants $P$ and $\Delta_S$ to coincide. However, as we mentioned in the introduction, computational data do not indicate that one invariant is topologically stronger than the other. A simple explanation would be that $P$ is a scalar multiple of  $\Delta_S$, that is, there exist $(c_n)_{n \in \mathbb{N}}$ in $\mathbb{C}(q,\zeta,u,z,E)$ such that
\begin{equation}\label{c-eq}
P(\widehat{\a}) = c_{n}\, \Delta_S(\widehat{\a}) \quad{(\a \in B_n)}
\end{equation}
for all $n \in \mathbb{N}$.
Then for the identity braid $1$ in each $B_n$ we have:
$$
P(\,\widehat{1}\,)=D_{\mathcal{H}}^{n-1}=c_n \,D_{\rm Y}^{n-1}=c_n \, \Delta_S(\,\widehat{1}\,).
$$
We deduce that
$$c_n=\frac{D_{\mathcal{H}}^{n-1}}{D_{\rm Y}^{n-1}}=\left(\frac{D_{\mathcal{H}}}{D_{\rm Y}}\right)^{n-1}.$$
Thus, if (\ref{c-eq}) holds, we must have
\begin{equation}
\frac{\tau(\pi(\alpha))}{{\rm tr}_{_S}(\delta(\alpha))}=\left(\frac{\zeta}{z}\right)^{\epsilon(\alpha)}
\end{equation}
for all $\alpha\in B_n$ and for all $n \in \mathbb{N}$.
Taking $\alpha=\sigma_1^{-1} \in B_n$, for $n\geq 2$, we obtain
\begin{equation}
{(u\zeta+z^2-uEz+Ez)} \, q={u\zeta(\zeta+1)},
\end{equation}
which in turn yields (see \S \ref{Some equalities})
$$D_{\mathcal{H}}=D_{\rm Y}.$$
We conclude that
$c_n=1$
for all $n \in \mathbb{N}$. Combining this with Theorem \ref{mainthm},  we obtain the following result:

\begin{thm}\label{notscalar}
Let $X_S$ be a solution of the {\rm E}-system. Let ${\rm tr}_{_S}$ be the corresponding specialized Juyumaya trace on ${\rm Y}_{d,n}(u)$ with parameter $z$, and let
$\tau$ be the Ocneanu trace on $\mathcal{H}_{n}(q)$ with parameter $\zeta$. Let $E={\rm tr}_{_S}(e_i)$ for all $i=1,\ldots,n-1$.
Then there exist $(c_n)_{n \in \mathbb{N}}$ in $\mathbb{C}(q,\zeta,u,z,E)$ such that, for all $n \in \mathbb{N}$,
$$P(\widehat{\a}) = c_{n}\, \Delta_S(\widehat{\a}) \quad{(\a \in B_n)}$$
if and only if
$P=\Delta_S$, that is, if and only if we are in one of the cases portrayed in the  table of Theorem \ref{mainthm}.
\end{thm}


\begin{thebibliography}{00}




\bibitem[CJJKL]{CJJKL} S. Chmutov, S. Jablan, J. Juyumaya, K. Karvounis, S. Lambropoulou, Computations on the Yokonuma--Hecke algebras, work in progress. See
\verb#http://www.math.ntua.gr/~sofia/yokonuma/index.html#.

\bibitem[GePf]{gp} M.~Geck, G.~Pfeiffer, Characters of finite Coxeter groups and Iwahori-Hecke algebras, London Math. Soc. Monographs, New Series 21, Oxford University Press, New York, 2000.

\bibitem[Jo]{jo} V.~F.~R.~Jones, {\em Hecke algebra representations of braid groups and link polynomials}, Annals of Math. {\bf 126} (1987),  no. 2, 335--388.

\bibitem[Ju]{ju} J.~Juyumaya, {\em Markov trace on the Yokonuma-Hecke algebra},
   J. Knot Theory and its Ramifications {\bf 13} (2004) 25--39.

\bibitem[JuLa1]{jula1} J.~Juyumaya, S.~Lambropoulou, {\em $p$-adic framed braids},
  Topology and its Applications {\bf 154} (2007) 1804--1826.

\bibitem[JuLa2]{jula2} J.~Juyumaya, S.~Lambropoulou, {\em $p$-adic  framed
   braids II}, to appear in Advances in Mathematics. See
   \verb#arXiv:0905.3626v2#.

\bibitem[JuLa3]{jula3} J.~Juyumaya, S.~Lambropoulou, {\em An adelic
   extension of the Jones polynomial},
   M. Banagl, D. Vogel (eds.) The mathematics of knots, Contributions
   in the Mathematical and Computational Sciences, Vol.~1, Springer.
  

\bibitem[JuLa4]{jula4} J.~Juyumaya, S.~Lambropoulou, {\em An invariant for
   singular knots},
   J. Knot Theory and its Ramifications, {\bf 18}(6) (2009) 825--840.


\bibitem[Yo]{yo} T.~Yokonuma, {\em Sur la structure des anneaux de Hecke
   d'un groupe de Chevalley fini},
   C.R. Acad. Sc. Paris, {\bf 264},  344--347 (1967).

\end{thebibliography}
\end{document}